\documentclass[twocolumn]{article}

\usepackage{arxiv}

\usepackage[utf8]{inputenc}
\usepackage[T1]{fontenc}
\usepackage{amsmath}
\usepackage{amsfonts}
\usepackage{amssymb}
\usepackage{amsthm}
\usepackage{mathtools}
\usepackage{tikz}
\usepackage{pgfplots}
\usepackage{algorithm}
\usepackage[noEnd]{algpseudocodex}
\usepackage[colorlinks=true, linkcolor=blue]{hyperref}
\usepackage{cleveref}
\usepackage{booktabs}
\usepackage{mdframed}
\usepackage{appendix}
\usepackage{caption}
\usepackage{subcaption}

\pgfplotsset{compat=1.17}

\newtheorem{thm}{Theorem}[section]
\newtheorem{rem}[thm]{Remark}
\newtheorem{lem}[thm]{Lemma}

\newtheorem{cor}[thm]{Corollary}

\allowdisplaybreaks%

\algrenewcommand\algorithmicrequire{\textbf{Input:}}
\algrenewcommand\algorithmicensure{\textbf{Output:}}

\numberwithin{equation}{section}

\usepackage{titlesec}
\titlespacing\section{0pt}{12pt plus 3pt minus 3pt}{1pt plus 1pt minus 1pt}
\titlespacing\subsection{0pt}{10pt plus 3pt minus 3pt}{1pt plus 1pt minus 1pt}
\titlespacing\subsubsection{0pt}{8pt plus 3pt minus 3pt}{1pt plus 1pt minus 1pt}

\input{macros.tex}
\let\le\leqslant%
\let\ge\geqslant%
\let\hat\widehat%
\let\tilde\widetilde%

\definecolor{lime}{HTML}{A6CE39}
\DeclareRobustCommand{\orcidicon}{
	\begin{tikzpicture}
	\draw[lime, fill=lime] (0,0)
	circle [radius=0.16]
	node[white] {{\fontfamily{qag}\selectfont \tiny ID}};
	\draw[white, fill=white] (-0.0625,0.095)
	circle [radius=0.007];
	\end{tikzpicture}
	\hspace{-2mm}
}
\foreach \x in {A, ..., Z} {%
  \expandafter\xdef\csname orcid\x\endcsname{\noexpand\href{https://orcid.org/\csname orcidauthor\x\endcsname}
    {\noexpand\orcidicon}}
}

\title{IRKA is a Riemannian Gradient Descent Method}

\usepackage{authblk}

\author[1]{Petar~Mlinari\'c\orcidA{}}
\author[2]{Christopher~A.~Beattie\orcidB{}}
\author[3]{Zlatko~Drma\v{c}\orcidC{}}
\author[4]{Serkan~Gugercin\orcidD{}}

\affil[1]{
  Department of Mathematics,
  Virginia Tech,
  Blacksburg,
  VA 24061
  (\texttt{mlinaric@vt.edu})
}
\affil[2]{
  Department of Mathematics,
  Virginia Tech,
  Blacksburg,
  VA 24061
  (\texttt{beattie@vt.edu})
}
\affil[3]{
  Faculty of Natural Sciences,
  University of Zagreb,
  10000 Zagreb,
  Croatia
  (\texttt{drmac@math.hr})
}
\affil[4]{
  Department of Mathematics and
  Division of Computational Modeling and Data Analytics,
  Academy of Data Science,
  Virginia Tech,
  Blacksburg,
  VA 24061
  (\texttt{gugercin@vt.edu})
}

\hypersetup{
  pdftitle={IRKA is a Riemannian Gradient Descent Method},
  pdfauthor={
    Petar Mlinari\'c,
    Christopher A. Beattie,
    Zlatko Drma\v{c},
    Serkan Gugercin
  },
}

\begin{document}

\twocolumn[ 
  \begin{@twocolumnfalse} 

\maketitle

\begin{abstract}
  The iterative rational Krylov algorithm (IRKA) is a commonly used fixed point
iteration developed to minimize the $\Htwo$ model order reduction error.
In this work, IRKA is recast as a Riemannian gradient descent method with a
fixed step size over the manifold of rational functions having fixed degree.
This interpretation motivates the development of a Riemannian gradient descent
method utilizing as a natural extension variable step size and line search.
Comparisons made between IRKA and this extension on a few examples demonstrate
significant benefits.

\end{abstract}


\vspace{0.35cm}

  \end{@twocolumnfalse} 
] 

\section{Introduction}

The \ac{irka} is a commonly used method for $\Htwo$-optimal \ac{mor} of
linear time-invariant systems~\cite{GugAB06,GugAB08,AntBG10}.
As usually conceived,
it is presented as a fixed point iteration based on
interpolatory $\Htwo$-optimality conditions;
we describe it in \Cref{sec:irka}.

Notably, there have also been approaches for $\Htwo$ \ac{mor} via
Riemannian optimization over various matrix manifolds
(see, e.g.,~\cite{YanL99,XuZ13,XuJY15,SatS15,BenCS19}).
In this work, we focus on the interpolatory \ac{irka} framework and
recast \ac{irka} as a Riemannian optimization method over
the manifold of stable rational functions of fixed McMillan degree
$\ManRatStab$, embedded in the Hardy space $\Hardy$.
This facilitates development of a geometric framework for \ac{irka} and
the interpretation of it as a Riemannian gradient descent method
with fixed step size.
We build on this observation, introduce a variable step size extension, and
explore the numerical performance.

These key results depend on the observation that the set of stable rational
functions of fixed McMillan degree $\ManRatStab$ is an embedded submanifold in
the Hardy space $\Hardy$.
We prove this while also drawing out issues related to connectedness of the
manifold and the Cauchy index.
An accessible reference for these facts in the continuous-time case
was not known to us.
This framework allows a new geometric interpretation of \ac{irka} as a
Riemannian optimization method with fixed step size.
This in turn leads to the introduction of an enhanced approach utilizing a
backtracking line search that preserves stability, and guarantees a decrease in
the $\Htwo$ error at each step.
Significantly, the utilization of different step sizes can be interpreted as a
fixed point iteration with an equivalent set of interpolatory conditions.

We begin with a background on $\Htwo$-optimal \ac{mor} and
Riemann optimization in \Cref{sec:irka,sec:riemann-opt}, respectively.
We then discuss known results on the manifold structure of $\ManRatStab$ and
introduce the new result for continuous-time systems in
\Cref{sec:rat-fun-manifold}.
This allows us to interpret the $\Htwo$-optimal \ac{mor} problem as a
Riemannian optimization problem and
derive the Riemannian gradient of the $\Htwo$ error in \Cref{sec:h2-riemann}.
Then we give a geometric interpretation of necessary optimality conditions and
\ac{irka} in \Cref{sec:geom},
leading to the interpretation of \ac{irka} as a Riemannian gradient descent
method with a fixed step size in \Cref{sec:irka-riemann}.
In \Cref{sec:algorithm}, we introduce a variant of \ac{irka},
called \ac{irka2},
using a variable step size that enforces stability preservation and
decreases the $\Htwo$ error at every step.
We compare the two methods in \Cref{sec:numerics} and
conclude with \Cref{sec:conclusion}.

\section{\texorpdfstring{$\Htwo$}{H2}-optimal Model Order Reduction}%
\label{sec:irka}

We briefly describe \ac{irka} and summarize relevant background and context
for later discussion.

\subsection{Model Order Reduction Problem}

Given \iac{fom} of order $n$
\begin{align*}
  E \dot{x}(t) & = A x(t) + B u(t),\ x(0) = 0, \\*
  y(t)         & = C x(t),
\end{align*}
the goal is to find \iac{rom} of order $r$
\begin{align*}
  \hE \dot{\hx}(t) & = \hA \hx(t) + \hB u(t),\ \hx(0) = 0, \\*
  \hy(t)           & = \hC \hx(t),
\end{align*}
such that $r \ll n$ and
the outputs $y(t), \hy(t) \in \Cp$ are close to one another
when the systems are presented with the same input $u(t) \in \Cm$.
Here, $x(t) \in \Cn$ is the full-order state and
$\hx(t) \in \Cr$ is the reduced-order state;
$E, A \in \Cnn$,
$B \in \Cnm$,
$C \in \Cpn$,
$\hE, \hA \in \Crr$,
$\hB \in \Crm$, and
$\hC \in \Cpr$.
We assume
$E$ and $\hE$ are invertible and that
$E^{-1} A$ and $\hE^{-1} \hA$ are Hurwitz
(their eigenvalues are in the open left half-plane).
We develop and present theory for general complex systems.
The case of real systems follow directly from our discussion and
we add remarks when necessary to highlight the real case.

To measure the distance between $y$ and $\hy$,
we use the transfer functions $H$ and $\hH$ of the \ac{fom} and \ac{rom},
respectively:
\begin{equation*}
  H(s) = C {(s E - A)}^{-1} B
  \quad \textnormal{and} \quad
  \hH(s) = \hC \myparen*{s \hE - \hA}^{-1} \hB.
\end{equation*}
The $\Htwo$ error is defined as
\begin{equation*}
  \normHtwo*{H - \hH}
  = \myparen*{
    \frac{1}{2 \pi}
    \int_{-\infty}^{\infty}
    \normF*{H(\imag \omega) - \hH(\imag \omega)}^2
    \dif{\omega}
  }^{1/2}.
\end{equation*}
A uniform bound on the output error can be formulated:
\begin{equation*}
  \normLinf*{y - \hy}
  \le
  \normHtwo*{H - \hH}
  \normLtwo{u}.
\end{equation*}
This motivates reducing the $\Htwo$ error, thus assuring the output error
$\normLinf{y - \hy}$ to be small for bounded input energy.
This leads one to the $\Htwo$-optimal \ac{mor} problem
\begin{equation}\label{eq:h2-opt-prob}
  \minimize_{
    \substack{
      \hH \textnormal{ of order } r \\
      \hH \textnormal{ stable}
    }
  }
  \quad
  \normHtwo*{H - \hH},
\end{equation}
where stability of $\hH$ means that its poles are in the open left half-plane.
Our focus here is on interpolatory and $\Htwo$-optimal \ac{mor}.
For further approaches, see, e.g.,~\cite{ZhoDG96,Ant05,BenOCW17,AntBG20,%
  BenGQ+21a,BenGQ+21b}.

\subsection{Hardy Space}

Under the assumptions made for the \ac{fom} and \ac{rom},
both $H$ and $\hH$ are elements of the Hardy space
$\Hardy = \Htwo^{p \times m}(\bbC^+)$ (see, e.g.,~\cite[Section~5.1.3]{Ant05}),
where $\bbC^+$ denotes the open right half-plane.
$\Hardy$ is a Hilbert space of analytic functions
$\fundef{F}{\bbC^+}{\Cpm}$ satisfying
\(
\sup_{\xi > 0}
\int_{-\infty}^{\infty}
\normF{F(\xi + \imag \omega)}^2
\dif{\omega}
< \infty
\)
with the inner product
\[
  \ipHtwo{F}{G}
  =
  \frac{1}{2 \pi}
  \int_{-\infty}^{\infty}
  \trace{F(\imag \omega)\herm G(\imag \omega)}
  \dif{\omega},
\]
using extensions of $F$ and $G$ from $\bbC^+$ to $\overline{\bbC^+}$.
We have particular interest in the real Hilbert space of
real $p \times m$ transfer functions,
$\ReHardy \subset \Hardy$,
\begin{equation*}
  \ReHardy
  = \mybrace*{
    H \in \Hardy :
    \overline{H(s)} = H(\overline{s}),
    \forall s \in \bbC^+
  }.
\end{equation*}

\subsection{Interpolatory Necessary Conditions for Optimality}

The following theorem gives necessary $\Htwo$-optimality conditions when the
\ac{rom} has simple poles (see, e.g.,~\cite{MeiL67,AntBG10}).
\begin{thm}
  Let $H, \hH \in \ReHardy$ be such that $\hH$ has the pole-residue form
  \[\hH(s) = \sum_{i = 1}^r \frac{c_i b_i\herm}{s - \lambda_i},\]
  where $\lambda_i$ are pairwise distinct.
  Let $\hH$ be an $\Htwo$-optimal approximation for $H$ as
  in~\eqref{eq:h2-opt-prob}.
  Then for $i = 1, 2, \ldots, r$
  \begin{subequations}\label{eq:h2-opt-cond-interp}
    \begin{align}
      H\myparen*{-\overline{\lambda_i}} b_i
       & =
      \hH\myparen*{-\overline{\lambda_i}} b_i,      \\
      c_i\herm H\myparen*{-\overline{\lambda_i}}
       & =
      c_i\herm \hH\myparen*{-\overline{\lambda_i}}, \\
      c_i\herm H'\myparen*{-\overline{\lambda_i}} b_i
       & =
      c_i\herm \hH'\myparen*{-\overline{\lambda_i}} b_i.
    \end{align}
  \end{subequations}
\end{thm}
That is, the $\Htwo$-optimal \acp{rom} are bitangential Hermite interpolants
at reflected poles of the \ac{rom}
in tangent directions given by the (vector) residues of the \ac{rom}.
This theorem directly extends from $\ReHardy$ to $\Hardy$ as well.
Extension to higher-order poles is possible~\cite{VanGA10},
but we focus on simple poles in the theoretical analysis for simplicity.

\subsection{Bitangential Hermite Interpolation}%
\label{sec:bhi}

Given interpolation points $\sigma_i \in \bbC$ and
tangential directions $b_i \in \Cm$, $c_i \in \Cp$,
for $i = 1, 2, \ldots, r$,
bitangential Hermite interpolants (as required for $\Htwo$-optimality)
can be constructed via Petrov-Galerkin projection based on rational Krylov
subspaces~\cite{GugAB06}:
For $H(s) = C {(s E - A)}^{-1} B$,
define projection matrices, $V, W \in \Cnr$, such that
\begin{subequations}\label{eq:spanVW}
  \begin{align}
    \label{eq:spanV}
    \image{V}
     & = \myspan*{\myparen*{\sigma_i E - A}^{-1} B b_i}_{i = 1}^r,      \\
    \image{W}
     & = \myspan*{\myparen*{\sigma_i E - A}\mherm C\herm c_i}_{i = 1}^r
  \end{align}
\end{subequations}
($\image{M}$ denotes the \emph{image} or \emph{range} of $M$).
Then \iac{rom} is constructed as
\begin{equation}\label{eq:pg}
  \hE = W\herm E V,\
  \hA = W\herm A V,\
  \hB = W\herm B,\
  \hC = C V.
\end{equation}
For $\hH(s) = \hC {(s \hE - \hA)}^{-1} \hB$,
assuming invertibility where needed, we have
for $i = 1, 2, \ldots, r$~\cite{AntBG10}:
\begin{align*}
  H(\sigma_i) b_i           & = \hH(\sigma_i) b_i,           \\
  c_i\herm H(\sigma_i)      & = c_i\herm \hH(\sigma_i),      \\
  c_i\herm H'(\sigma_i) b_i & = c_i\herm \hH'(\sigma_i) b_i.
\end{align*}
Note that $V$ and $W$ can be computed from solutions to Sylvester equations
as we discuss next.
Given interpolation data $\{\sigma_i, b_i, c_i\}$,
let the $i$th column of $V$ be ${(\sigma_i E - A)}^{-1} B b_i$.
This construction satisfies~\eqref{eq:spanV}.
Then it directly follows that $V$ satisfies the Sylvester equation
\begin{equation*}
  A V - E V \Sigma + B \tR\herm = 0,
\end{equation*}
where $\Sigma = \mydiag{\sigma_i}$ and the $i$th row of $\tR$ is $b_i\herm$.
Then, for any invertible $T, \tE \in \Crr$, we have
\begin{equation*}
  \myparen*{
    A V
    - E V \Sigma
    + B \tR\herm
  }
  T \tE\herm
  = 0.
\end{equation*}
Therefore,
\begin{equation*}
  A (V T) \tE\herm
  + E (V T) \myparen*{-T^{-1} \Sigma T \tE\herm}
  + B \myparen*{\tR\herm T \tE\herm} = 0.
\end{equation*}
With
$\tV = V T$,
$\tA = -\tE T\herm \Sigma\herm T\mherm$, and
$\tB = \tE T\herm \tR$,
we find
\begin{equation}\label{eq:V-sylv}
  A \tV \tE\herm
  + E \tV \tA\herm
  + B \tB\herm = 0.
\end{equation}
Note that $\image{\tV} = \image{V}$ and
$\tE^{-1} \tA$ has $-\overline{\sigma_i}$ as eigenvalues.
Analogous constructions hold for $W$, thus, as claimed,
both $V$ and $W$ can be constructed from solving Sylvester equations of the
form~\eqref{eq:V-sylv}.
This fact will be employed in the algorithmic development
in~\Cref{sec:algorithm}.
Note that solving~\eqref{eq:V-sylv} does not require that $\tE^{-1} \tA$ has
simple eigenvalues~\cite{BenKS11}.

Assuming that the interpolation data is closed under conjugation,
$V$ and $W$ can be chosen to be real so that the \ac{rom} is also real.
See~\cite{AntBG20} for further details.

\subsection{The IRKA Iteration}

Construction of $V$ and $W$ above assumed the knowledge of interpolation data.
However, as shown in~\eqref{eq:h2-opt-cond-interp},
the interpolation data for $\Htwo$ optimality depend on the optimal \ac{rom} to
be constructed.
Thus, satisfying~\eqref{eq:h2-opt-cond-interp} requires an iterative process.
This observation, together with the interpolatory projection framework,
led to the development of the \emph{\acf{irka}}~\cite{GugAB06,GugAB08,AntBG10},
a fixed point iteration algorithm that, given an iterate
$\hH_k(s) = \sum_{i = 1}^r \frac{\cik \bikstar}{s - \lamik}$,
determines the next iterate $\hH_{k + 1}$ from the interpolation conditions,
for $i = 1, 2, \ldots, r$:
\begin{subequations}\label{eq:irka}
  \begin{align}
    H\myparen*{-\lamikstar} \bik
     & =
    \hH_{k + 1}\myparen*{-\lamikstar} \bik,     \\
    \cikstar H\myparen*{-\lamikstar}
     & =
    \cikstar \hH_{k + 1}\myparen*{-\lamikstar}, \\
    \cikstar H'\myparen*{-\lamikstar} \bik
     & =
    \cikstar \hH_{k + 1}'\myparen*{-\lamikstar} \bik.
  \end{align}
\end{subequations}
Using Petrov-Galerkin projection as described in~\ref{sec:bhi},
these conditions uniquely determine a rational transfer function of order $r$.
Then, this fixed point iteration, i.e., \ac{irka},
is run until convergence upon which the \ac{rom} satisfies the interpolatory
conditions~\eqref{eq:h2-opt-cond-interp}.
For details of \ac{irka}, we refer the reader to~\cite{GugAB08,AntBG20}.
Despite its success in practice and its usage across numerous applications,
since \ac{irka} is a fixed point iteration,
the convergence and stability of the \ac{rom} are not theoretically guaranteed.
In this work, by recasting \ac{irka} in a Riemannian optimization setting,
we develop efficient variants that aim to resolve these issues.

\section{Riemannian Optimization Background}%
\label{sec:riemann-opt}

We summarize here a few topics from Riemannian optimization
(see also~\cite{AbsMS08,Bou23}),
focusing, in particular, on Riemannian submanifolds
(see, e.g.,~\cite{Lan95,Kli11,Lee12}),
that will be used heavily in the rest of the paper.

\subsection{General Manifolds}

A \emph{(smooth) manifold} $\cM = (M, \cA)$,
modeled on a Banach space $\bbE$,
is a set $M$ together with an \emph{atlas} $\cA$,
which is a collection of \emph{charts} $(U_i, \varphi_i)$
($i$ is an element of an index set)
such that
\begin{enumerate}
  \item each $U_i$ is a subset of $M$ and
        the $U_i$ cover $M$,
  \item each $\varphi_i$ is a bijection from $U_i$
        to an open subset $\varphi(U_i)$ of $\bbE$ and
        $\varphi_i(U_i \cap U_j)$ is open in $\bbE$ for any $i, j$,
  \item the map
        $\fundef{
            \varphi_j \circ \varphi_i^{-1}
          }{
            \varphi_i(U_i \cap U_j)
          }{
            \varphi_j(U_i \cap U_j)
          }$
        is smooth (infinitely differentiable) for each pair of indices $i, j$.
\end{enumerate}
We additionally assume that the atlas topology
defined by the maximal atlas $\cA^+$
(the collection of all charts compatible with elements of $\cA$)
is Hausdorff and second-countable.
As is common in the literature, we conflate
$\cM$ with $M$ and
$(U_i, \varphi_i)$ with $\varphi_i$
in the following.

The \emph{dimension} of a manifold, denoted by $\dim(\cM)$,
is defined to be $\dim(\bbE)$.
We are interested in both finite and infinite-dimensional manifolds.
In particular, note that the Hardy space $\Hardy$ is a manifold with
$\bbE = \Hardy$ and
the identity map $\fundef{\operatorname{id}}{\Hardy}{\Hardy}$ as the chart.
Additionally, we are interested in embedded submanifolds
(which we discuss in~\Cref{{sec:embedsub}}).
But, first, we need a general concept of a smooth function and its differential.

\subsection{General Tangent Space and Smooth Functions}

A function $\fundef{F}{\cM}{\cN}$ from a manifold $\cM$ to a manifold $\cN$ is
\emph{smooth at $p$} if
$\fundef{\psi \circ F \circ \varphi^{-1}}{\varphi(U)}{\psi(V)}$
is smooth for a chart $(U, \varphi)$ of $\cM$ and a chart $(V, \psi)$ of $\cN$
such that $p \in U$ and $F(p) \in V$.
A function is \emph{smooth} if it is smooth at every point.

For a manifold $\cM$ and point $p \in \cM$,
let $\cC$ be the set of smooth curves $\fundef{c}{I}{\cM}$
such that $c(0) = p$
where $I \subseteq \bbR$ is an interval around zero.
Define an equivalence relation $\sim$ between two curves $c_1, c_2$ by
\begin{equation*}
  c_1 \sim c_2
  \quad \Leftrightarrow \quad
  (\varphi \circ c_1)'(0)
  = (\varphi \circ c_2)'(0)
\end{equation*}
for some chart $\varphi$ of $\cM$ around $p$.
The \emph{tangent vectors} of $\cM$ at $p$ are then equivalence classes $[c]$
under the equivalence relation $\sim$.
The \emph{tangent space} of $\cM$ at $p$,
denoted $\T_p\cM$,
is the set of all tangent vectors at $p$, i.e.,
$\T_p\cM = \cC / {\sim}$.
A tangent space is a vector space of the same dimension as the manifold.

The \emph{differential} of a smooth function $\fundef{F}{\cM}{\cN}$
at $p \in \cM$ is the linear operator
$\fundef{DF(p)}{\T_p\cM}{\T_{F(p)}\cN}$,
defined by $DF(p)[v] = [F \circ c]$,
where $c$ is a representative curve, i.e., $v = [c] \in \T_p\cM$.
Note that the brackets in $[v]$ are only used to denote
that $DF(p)$ is applied to $v$.

\subsection{Embedded Submanifolds}%
\label{sec:embedsub}

A \emph{smooth immersion} is a smooth mapping such that
its differential is injective at each point and
the image of the differential is a closed subspace.
A \emph{smooth embedding} is a smooth immersion that
is also a homeomorphism onto its image, i.e.,
its inverse is continuous.

An \emph{embedded submanifold} of $\overline{\cM}$ is a subset
$\cM \subseteq \overline{\cM}$ that is a manifold in the subspace topology,
endowed with a smooth structure with respect to which the inclusion map
$\cM \hookrightarrow \overline{\cM}$ is a smooth embedding.
$\overline{\cM}$ is called an \emph{ambient space}.

A function $\fundef{F}{\cM}{\cN}$ between embedded submanifolds is smooth
if and only if $F$ can be extended to a smooth function over
an open neighborhood of $\cM$ in its ambient space.

\subsection{Tangent Vectors, Tangent Spaces, Tangent Bundle, and Normal Spaces
  of Embedded Submanifolds}

Let $\cM$ be an embedded submanifold of a Banach space $\bbE$.
The tangent space of $\cM$ at $p \in \cM$
can be identified with
the set of velocities at $p$ of all smooth curves in $\cM$ going through $p$,
i.e., $\T_p\cM \simeq \{c'(0) \mid \fundef{c}{I}{\cM},\ c(0) = p\}$,
where $I \subseteq \bbR$ is any interval containing zero.
In the following, we always interpret the tangent space of
an embedded submanifold of a Banach space
as a subset of the ambient space.
Then, for every $p \in \cM$,
$\T_p\cM$ forms a subspace of $\bbE$
of the same dimension as the dimension of the manifold.

The \emph{tangent bundle}, denoted by $\T\cM$, is a disjoint union of all
tangent spaces of $\cM$, in particular,
$\T\cM = \mybrace{(p, v) \mid p \in \cM \text{ and } v \in \T_p\cM}$.
The tangent bundle is itself a manifold of dimension $2 \dim(\cM)$.

If $\bbE$ is additionally a Hilbert space,
the \emph{normal space} of $\cM$ at $p \in \cM$, denoted $\N_p\cM$,
is the orthogonal complement of the tangent space at $p$ in $\bbE$,
i.e., $\N_p\cM = {(\T_p\cM)}^{\perp}$.

A \emph{scalar field} over $\cM$ is a function $\fundef{f}{\cM}{\bbR}$.
A \emph{vector field} over $\cM$ is a function $\fundef{X}{\cM}{\T\cM}$
such that $X(p) \in \T_p\cM$ for all $p \in \cM$.

\subsection{Riemannian Manifolds and Submanifolds}

A \emph{metric} over $\cM$ is a collection of inner products
$\ip{\cdot}{\cdot}_p$ over $\T_p\cM$ for all $p \in \cM$.
A \emph{Riemannian metric} is a metric $\ip{\cdot}{\cdot}_p$ such that
$p \mapsto \ip{X(p)}{Y(p)}_p$ is a smooth scalar field
for all smooth vector fields $X, Y$.
A \emph{Riemannian manifold} is a manifold with a Riemannian metric.

A \emph{Riemannian submanifold} is an embedded submanifold
that inherits the Riemannian metric from the ambient manifold.
Note that it is always possible for an embedded submanifold of a Riemannian
manifold to inherit the metric because,
if $\cM \subseteq \overline{\cM}$,
then $\T_p\cM \subseteq \T_p\overline{\cM}$ for all $p \in \cM$.

\subsection{Riemannian Optimization}

Now we have all the necessary background to summarize the basics of Riemannian
optimization.
A \emph{Riemannian optimization problem} is an optimization problem of the form
\begin{equation*}
  \minimize_{p \in \cM} f(p),
\end{equation*}
where $\cM$ is a Riemannian manifold and $f$ is a smooth scalar field.
The \emph{gradient} at $p \in \cM$ of a smooth scalar field $f$ on a
Riemannian manifold $\cM$ is the unique tangent vector $\grad f(p)$ in $\T_p\cM$
such that
\begin{equation*}
  Df(p)[v] = \ip{\grad f(p)}{v}_p
\end{equation*}
for all $v \in \T_p\cM$.
If $\cM$ is a Riemannian submanifold of a Hilbert space $\bbE$,
the gradient can be obtained via
\begin{equation}\label{eq:riemann-grad-proj}
  \grad f(p)
  = \Proj_p\myparen*{\nabla \of(p)},
\end{equation}
where
$\fundef{\Proj_p}{\bbE}{\T_p\cM}$ is the orthogonal projector onto $\T_p\cM$,
$\of$ is a smooth extension of $f$ to an open neighborhood of $\cM$ in $\bbE$,
and $\nabla \of(p)$ is the Euclidean gradient of $\of$ at $p$.

As for Euclidean optimization,
the first-order necessary optimality condition for $p$ to be a local minimum is
\begin{equation}\label{eq:riemann-grad-zero}
  \grad f(p) = 0,
\end{equation}
If $\cM$ is an Riemannian submanifold, using~\eqref{eq:riemann-grad-proj},
we have that the necessary optimality condition~\eqref{eq:riemann-grad-zero} is
equivalent to
\begin{equation}\label{eq:riemann-orth}
  \nabla \of(p) \perp \T_p\cM.
\end{equation}
A step of Riemannian gradient descent takes the form
\begin{equation}\label{eq:rgd}
  p_{k + 1} = \R_{p_k}\myparen*{-\alpha_k \grad f(p_k)},
\end{equation}
where $\fundef{\R_{p_k}}{\T_{p_k}\cM}{\cM}$ is a \emph{retraction} and
$\alpha_k$ are positive.
In simple terms,
a step of Riemannian gradient descent starts from the current iterate $p_k$,
moves in the direction of the negative gradient at $p_k$, and
then returns to the manifold $\cM$ via retraction.

Generally, a retraction $\R$ is a mapping from the tangent bundle $\T\cM$
to the manifold $\cM$ such that $\R_p$, the restriction of $\R$ to $\T_p\cM$,
is ``locally rigid'' around zero, i.e.,
\begin{equation*}
  \R_p(0) = p
  \ \text{ and } \
  D\R_p(0) = \operatorname{id},
\end{equation*}
where $\fundef{\operatorname{id}}{\T_p\cM}{\T_p\cM}$ is the identity map.
The definition of a retraction can be relaxed so as to be defined only over an
open subset of the tangent bundle containing all of the tangent zero vectors
i.e., containing $(p, 0)$ for all $p \in \cM$.
We employ these Riemannian optimization concepts,
especially~\eqref{eq:riemann-grad-proj}--\eqref{eq:rgd},
in the theoretical and algorithmic developments presented
in the following sections.

\section{Rational Functions Form a Manifold}%
\label{sec:rat-fun-manifold}

Let $\ManRatStab$ be the set of stable rational functions:
\begin{equation*}
  \ManRatStab = \mybrace*{\hC \myparen*{s I - \hA}^{-1} \hB :
    \begin{array}{@{}l@{}}
      \hA \in \Crr \textnormal{ is Hurwitz}, \\
      \hB \in \Crm,\
      \hC \in \Cpr,                          \\
      \myparen*{\hA, \hB, \hC} \text{ is minimal}
    \end{array}
  },
\end{equation*}
where minimality means that the realization $\myparen{\hA, \hB, \hC}$ is of the
least possible order.
We know that $\ManRatStab$ is a subset of $\Hardy$.
In this section,
we prove that $\ManRatStab$ is an embedded submanifold of $\Hardy$,
an important result  employed
in later sections.
Additionally, we show that $\ReManRatStab$,
the set of real stable rational functions,
is an embedded submanifold of $\ReHardy$.

\subsection{Previous Work on Manifolds of Rational Functions}%
\label{sec:manifold}

Let $\ReManMatStab$ be the set of stable, minimal state-space realizations,
i.e., $\ReManMatStab \subset \Rrr \times \Rrm \times \Rpr$ and
$\myparen{\hA, \hB, \hC} \in \ReManMatStab$
where
$\hA$ is Hurwitz and
$\myparen{\hA, \hB, \hC}$ is minimal.
Similarly, we use $\ReManRat$ and $\ReManMat$ to denote the sets corresponding
to, respectively, $\ReManRatStab$ and $\ReManMatStab$
but without the stability assumptions.

For the \ac{siso} case ($m = p = 1$),
Brockett~\cite{Bro76} identified the elements of $\ReManRat$ by the
$2 r$ coefficients in the ``polynomial over monic polynomial'' form and
showed that this set forms an embedded submanifold of $\bbR^{2 r}$
consisting of $r + 1$ connected components.
As a corollary to~\cite[Theorem~1]{Bro76},
one obtains the same result for stable rational functions.
Additionally, we find that $\ManRatStab$ is connected for $m = p = 1$.

Hazewinkel~\cite[Theorem~2.5.17]{Haz77} constructed a manifold structure over
the quotient set $\ReManMat / {\sim}$
(based on Hazewinkel and Kalman~\cite{HazK76})
where $\sim$ is the state-space equivalence relation, i.e.,
$\myparen{\hA_1, \hB_1, \hC_1}
  \sim
  \myparen{\hA_2, \hB_2, \hC_2}$
if and only if
there is an invertible matrix $T \in \Rrr$ such that
\begin{equation*}
  \hA_1 = T \hA_2 T^{-1},\
  \hB_1 = T \hB_2,\
  \hC_1 = \hC_2 T^{-1}.
\end{equation*}
Delchamps asserts
$\ReManRat$ (and $\ManRat$) are analytic manifolds of
(complex) dimension $r (m + p)$
(\cite[Theorem~2.1]{Del85};
Delchamps attributes the result to Clark~\cite{Cla76}).
Byrnes and Duncan~\cite{ByrD82} show in the proof of Proposition~A.7 that
$\ReManRat$ is connected when $\max(m, p) > 1$.
The proof can be extended to stable rational functions in $\ReManRatStab$ and
to complex rational functions in $\ManRatStab$.

\subsection{Stable Systems Form an Embedded Submanifold}

Our first main result is to establish the manifold structure of $\ManRatStab$
as an embedded submanifold of the Hardy space $\Hardy$
(in contrast to earlier results which establish a manifold
structure over different spaces).
\begin{thm}\label{thm:submanifold}
  The manifold $\ManRatStab$ ($\ReManRatStab$) is an embedded submanifold of
  $\Hardy$ ($\ReHardy$).
\end{thm}
\begin{proof}
  Exploiting the ideas from the proof of Theorem~2.1 in~\cite{AlpBG94} for
  discrete-time dynamical systems,
  we will prove the result by showing that the natural inclusion
  $\fundef{\immj}{\ManRatStab}{\Hardy}$
  is a smooth embedding where the manifold structure of $\ManRatStab$ is from
  Hazewinkel and Kalman~\cite{HazK76}.
  We show that $\immj$ is an embedding, i.e.,
  \begin{enumerate}
    \item $\immj$ is differentiable,
    \item the differential $\D \immj (\hH)$ is injective and
          its image $\image{\D \immj(\hH)}$ is closed for all $\hH \in \ManRatStab$,
          and
    \item $\immj^{-1}$ is continuous.
  \end{enumerate}
  Define $\fundef{\Pi}{\ManMatStab}{\ManRatStab}$ by
  $\Pi(\hA, \hB, \hC)[s] = \hC {(s I - \hA)}^{-1}\!\hB$.
  Using the charts $\fundef{\varphi_i}{U_i}{\ManMatStab}$ of $\ManRatStab$
  from~\cite{HazK76},
  $\immj$ can be locally expressed as $\immj \circ \Pi \circ \varphi_i$.
  Since $\varphi_i$ is smooth by definition,
  we need only show that $\fundef{\immj \circ \Pi}{\ManMatStab}{\Hardy}$
  is smooth.
  We see that
  \begin{align*}
     &
    \D_{\hA} (\immj \circ \Pi)\myparen*{\hA, \hB, \hC}[X]
    = -\hC \myparen*{s I - \hA}^{-1} X \myparen*{s I - \hA}^{-1} \hB, \\
     &
    \D_{\hB} (\immj \circ \Pi)\myparen*{\hA, \hB, \hC}[W]
    = \hC \myparen*{s I - \hA}^{-1} W,                                \\
     &
    \D_{\hC} (\immj \circ \Pi)\myparen*{\hA, \hB, \hC}[V]
    = V \myparen*{s I - \hA}^{-1} \hB,
  \end{align*}
  therefore, $\immj$ is smooth, proving the first assertion.

  For the second assertion,
  note that $\image{\D \immj(\hH)}$ consists of rational matrix-valued functions
  having the form
  \begin{align*}
     &
    V \myparen*{s I - \hA}^{-1} \hB
    + \hC \myparen*{s I - \hA}^{-1} W \\*
     &
    - \hC \myparen*{s I - \hA}^{-1} X \myparen*{s I - \hA}^{-1} \hB
  \end{align*}
  for arbitrary $X \in \Cnn$, $W \in \Cnm$, and $V \in \Cpn$.
  Since the domain of $\D \immj(\hH)$ is $\T_{\hH}\ManRatStab$,
  it follows that $\dim\image{\D \immj(\hH)} \le \dim\ManRatStab = r (m + p)$.
  Using the same procedure as in~\cite{AlpBG94},
  we can find a linearly independent set with $r (m + p)$ elements,
  from which we conclude that $\D \immj(\hH)$ is injective.
  The image of $\D \immj(\hH)$ is closed since it is finite-dimensional,
  which proves the second assertion on $\immj$.

  For the third assertion,
  we show that $\immj^{-1}$ is continuous
  by showing the convergence of Markov parameters
  induced by convergence of $(\hH_k)$ in $\Hardy$
  for continuous-time dynamical system.
  Let $(\hH_k)$ be a convergent sequence in the image of $\immj$
  with $\hH$ as the limit.
  Let $M_i^{(k)}$ and $M_i$ denote the $i$th Markov parameter of
  $\hH_k$ and $\hH$, respectively.
  The Markov parameters $M_i^{(k)}$ and $M_i$ can be determined,
  via contour integrals, as
  \begin{align*}
    M_i^{(k)}
    =
    \frac{1}{2 \pi \imag}
    \oint_{\Gamma_R}
    s^i
    \hH_k(s)
    \dif{s},\
    M_i
    =
    \frac{1}{2 \pi \imag}
    \oint_{\Gamma_R}
    s^i
    \hH(s)
    \dif{s},
  \end{align*}
  where
  $\Gamma_R
    =
    [-\imag R, \imag R]
    \cup \{R e^{\imag \omega} : \omega \in (\frac{\pi}{2}, \frac{3 \pi}{2})\}$
  and $R > 0$ is big enough such that
  $\Gamma_R$ encloses all of the poles of $\hH_k$ and $\hH$,
  for all $k \ge 1$.
  Then we have that
  \begin{align*}
     & \normtwo*{M_i^{(k)} - M_i} \\
     & \qquad  \le
    \frac{1}{2 \pi}
    \int_a^b
    \abs{\gamma(t)}^i
    \normtwo*{\hH_k(\gamma(t)) - \hH(\gamma(t))}
    \abs{\gamma'(t)}
    \dif{t},
  \end{align*}
  where $\fundef{\gamma}{[a, b]}{\bbC}$ is a parametrization of $\Gamma_R$.
  The Cauchy-Schwarz inequality gives
  \begin{align*}
     &
    \normtwo*{M_i^{(k)} - M_i}
    \le
    \frac{1}{2 \pi}
    \myparen*{
      \int_a^b
      \abs{\gamma(t)}^{2 i}
      \abs{\gamma'(t)}^2
      \dif{t}
    }^{1/2}          \\*
     & \qquad \times
    \myparen*{
      \int_a^b
      \normtwo*{\hH_k(\gamma(t)) - \hH(\gamma(t))}^2
      \dif{t}
    }^{1/2}.
  \end{align*}
  Since convergence of the sequence of rational functions $\hH_k$ of bounded
  degree in $\Hardy$ implies $\mathcal{L}_2$ convergence on $\Gamma_R$,
  we conclude $M_i^{(k)}$ converges to $M_i$.

  Finally, it follows similarly that $\ReManRatStab$ is an embedded submanifold
  of $\ReHardy$.
\end{proof}

\section{\texorpdfstring{$\Htwo$}{H2} Minimization as a Riemannian
  Optimization}%
\label{sec:h2-riemann}

Theorem~\ref{thm:submanifold} shows that $\ManRatStab$ is an embedded
submanifold of $\Hardy$.
Building on that result, we now give an interpretation of the $\Htwo$-optimal
\ac{mor} problem~\eqref{eq:h2-opt-prob} as a Riemannian optimization problem.
We then compute the Riemannian gradient in this setting,
which will be used to derive necessary optimality conditions and
facilitate the interpretation of \ac{irka} as a Riemannian gradient descent
method.

First, we establish $\ManRatStab$ as a Riemannian manifold.
Clearly, it is possible to define a Riemannian metric over $\ManRatStab$
by inheriting the inner product from the ambient $\Hardy$ space.
In this way,
$\ManRatStab$ forms a Riemannian submanifold and
the $\Htwo$-optimal \ac{mor} problem~\eqref{eq:h2-opt-prob} can be written as
a Riemannian optimization problem
\begin{equation}\label{eq:h2-opt-prob-riemann}
  \minimize_{\hH \in \ManRatStab}
  f(\hH) = \normHtwo*{H - \hH}^2.
\end{equation}
To find $\grad f(\hH)$ using~\eqref{eq:riemann-grad-proj},
we first need to find a smooth extension of $f$
to an open neighborhood of $\ManRatStab$ in $\Hardy$.
To do so, we define $\fundef{\of}{\Hardy}{\bbR}$ by
$\of(G) = \normHtwo{H - G}^2$,
which satisfies $\of|_{\ManRatStab} = f$.
To find the Euclidean gradient of the extension $\of$, note that
\begin{align*}
  \of(G + \Delta G)
   & =
  \of(G)
  + 2 \Real{\ipHtwo{G - H}{\Delta G}}
  + \normHtwo{\Delta G}^2.
\end{align*}
If we restrict to real rational functions,
then we directly find that
$\nabla \of|_{\ReHardy}(G) = 2 (G - H)$ for all $G \in \ReHardy$.
Since $\fundef{\of}{\Hardy}{\bbR}$ is not complex differentiable
(the only real-valued analytic functions are constant functions),
we need a different concept of a complex gradient.
Following~\cite{NisAP08} (based on Wirtinger calculus), we use
$2 \nabla_{\overline{G}} \of
  = \nabla_{G^{\Re}} \of + \imag \nabla_{G^{\Im}} \of$,
where $G = G^{\Re} + \imag G^{\Im}$ and $G^{\Re}, G^{\Im}$ are real.
We find that
\begin{align*}
  \Real{\ipHtwo{G - H}{\Delta G}} = {}
   & \ipHtwo*{G^{\Re} - H^{\Re}}{\Delta G^{\Re}}    \\
   & + \ipHtwo*{G^{\Im} - H^{\Im}}{\Delta G^{\Im}}.
\end{align*}
Thus,
$\nabla_{G^{\Re}} \of(G^{\Re}) = 2 (G^{\Re} - H^{\Re})$ and
$\nabla_{G^{\Im}} \of(G^{\Im}) = 2 (G^{\Im} - H^{\Im})$, and
\begin{equation}\label{eq:h2-ext-grad}
  \nabla_{\overline{G}} \of(G) = G - H.
\end{equation}
Then, from~\eqref{eq:riemann-grad-proj},
it follows that the Riemannian gradient of the $\Htwo$-minimization problem is
given by
\begin{equation}\label{eq:riemann-grad-h2}
  \grad f(\hH)
  = \Proj_{\hH}\myparen*{\nabla_{\overline{\hH}} \of(\hH)}
  = \Proj_{\hH}\myparen*{\hH - H},
\end{equation}
where $\fundef{\Proj_{\hH}}{\Hardy}{\T_{\hH}\ManRatStab}$ is the orthogonal
projector onto $\T_{\hH}\ManRatStab$.

\section{Geometric Interpretation of Necessary Optimality Conditions and IRKA}%
\label{sec:geom}

Meier and Luenberger~\cite{MeiL66} derived
the interpolatory $\Htwo$-optimality conditions~\eqref{eq:h2-opt-cond-interp}
and their geometric interpretation
for the \ac{siso} case, i.e., $m = p = 1$.
Using our analysis from \Cref{sec:h2-riemann},
we generalize these results to \ac{mimo} systems using
Riemannian optimization theory.
In particular, in \Cref{sec:geointerpolate},
we show that any (locally) $\Htwo$-optimal \ac{rom} $\hH$ is such that
the error system $H - \hH$ is orthogonal to the manifold $\ManRatStab$.
Next, we show that the orthogonality conditions
are equivalent to the known interpolatory necessary optimality
conditions~\eqref{eq:h2-opt-cond-interp}.
Then, in \Cref{sec:irka-geom},
these results lead to a geometric interpretation of \ac{irka}.
In turn, this leads to the interpretation of \ac{irka} as a Riemannian
gradient descent method.

\subsection{Geometry of Interpolatory Conditions}%
\label{sec:geointerpolate}

Having interpreted the $\Htwo$-minimization as
a Riemannian optimization problem,
it follows that the necessary optimality condition~\eqref{eq:riemann-orth},
in view of~\eqref{eq:h2-ext-grad}, becomes
\begin{equation}\label{eq:orth-tan-space-cond}
  H - \hH \perp \T_{\hH}\ManRatStab.
\end{equation}
This orthogonality result is illustrated in \Cref{fig:h2-opt-cond}.
Therefore, $\Htwo$-optimal \acp{rom} are necessarily such that
the error system is orthogonal to the manifold $\ManRatStab$.

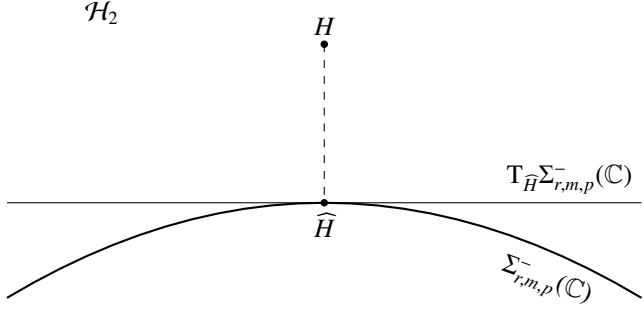
\begin{figure}[tb]
  \centering
  \begin{tikzpicture}[x=1em, y=1em, scale=1.2]
    \def\xmin{-10}
    \def\xmax{10}
    \def\coeff{-0.03}
    \def\H{5}

    \node at (-7, 6) {$\Hardy$};

    \coordinate (H) at (0, \H);
    \draw[fill] (H) circle (0.1) node[above] {$H$};

    \draw[domain=\xmin:\xmax, thick] plot (\x, \coeff*\x*\x);
    \node[rotate=-20] at (\xmax-3, \coeff*\xmax*\xmax+0.5) {$\ManRatStab$};

    \coordinate (Hr) at (0, 0);
    \draw[fill] (Hr) circle (0.1) node[below] {$\hH$};

    \draw (\xmin, 0) -- (\xmax, 0) node[above left] {$\T_{\hH}\ManRatStab$};

    \draw[dashed] (H) -- (Hr);
  \end{tikzpicture}
  \caption{
    Necessary $\Htwo$-optimality conditions in terms of orthogonality.
    $H$ is the full-order transfer function,
    $\ManRatStab$ is the manifold of stable rational functions of
    McMillan degree $r$,
    $\hH$ is the $\Htwo$-optimal reduced-order transfer function,
    $\T_{\hH} \ManRatStab$ is the tangent space of $\ManRatStab$ at $\hH$.
  }%
  \label{fig:h2-opt-cond}
\end{figure}

To derive the interpolatory necessary conditions
from~\eqref{eq:orth-tan-space-cond},
we first need to determine the tangent vectors of $\ManRatStab$ at $\hH$.
Our next result achieves this result for \ac{mimo} systems with simple poles.
\begin{lem}\label{lem:tan-space-span}
  Let $\hH \in \ManRatStab$,
  with a pole-residue form
  $\hH(s) = \sum_{i = 1}^r \frac{c_i b_i\herm}{s - \lambda_i}$,
  have pairwise distinct poles.
  Then the tangent space $\T_{\hH}\ManRatStab$ is spanned by
  \begin{equation}\label{eq:tan-space-span}
    \frac{e_j b_i\herm}{s - \lambda_i}, \ \
    \frac{c_i e_{\ell}\herm}{s - \lambda_i}, \ \
    \frac{c_i b_i\herm}{\myparen*{s - \lambda_i}^2},
  \end{equation}
  for $i = 1, 2, \ldots, r$, $j = 1, 2, \ldots, p$, $\ell = 1, 2, \ldots, m$.
\end{lem}
\begin{proof}
  Let $\fundef{\gamma}{I}{\ManRatStab}$ be a smooth curve
  where $0 \in I \subset \bbR$ and
  $\gamma(0) = \hH$.
  Since rational functions with simple poles form a relatively open subset of
  $\ManRatStab$,
  there is a neighborhood of $0$ such that $\gamma(t)$ has simple poles
  for all $t$ in that neighborhood.
  Let
  $\gamma(t)(s) = \sum_{i = 1}^r \frac{c_i(t) b_i(t)\herm}{s - \lambda_i(t)}$
  with
  $\lambda_i(0) = \lambda_i$,
  $b_i(0) = b_i$, and
  $c_i(0) = c_i$.
  Then
  \begin{align*}
    \gamma'(t)(s) {} =
     &
    \sum_{i = 1}^r
    \myparen*{
      \frac{c_i'(t) b_i(t)\herm}{s - \lambda_i(t)}
      + \frac{c_i(t) b_i'(t)\herm}{s - \lambda_i(t)}
      + \frac{c_i(t) b_i(t)\herm \lambda_i'(t)}{
        \myparen*{s - \lambda_i(t)}^2}
    }.
  \end{align*}
  Therefore,
  \begin{align*}
    \gamma'(0)(s) {} =
     &
    \sum_{i = 1}^r
    \myparen*{
      \frac{c_i'(0) b_i\herm}{s - \lambda_i}
      + \frac{c_i b_i'(0)\herm}{s - \lambda_i}
      + \frac{c_i b_i\herm \lambda_i'(0)}{
        \myparen*{s - \lambda_i}^2}
    },
  \end{align*}
  which is a linear combination of~\eqref{eq:tan-space-span},
  proving that all tangent vectors are in the span
  of~\eqref{eq:tan-space-span}.
  Conversely, since
  $\lambda_i'(0) \in \bbC$,
  $b_i'(0) \in \Cm$,
  $c_i'(0) \in \Cp$
  are arbitrary,
  it follows that all elements of the span of~\eqref{eq:tan-space-span} are
  tangent vectors.
\end{proof}
Now with the full characterization of the tangent vectors,
we can give an interpolatory interpretation of orthogonality.
\begin{thm}\label{thm:orth-interp}
  Let $\hH \in \ManRatStab$,
  with a pole-residue form
  $\hH(s) = \sum_{i = 1}^r \frac{c_i b_i\herm}{s - \lambda_i}$,
  have pairwise distinct poles.
  Furthermore, let $F \in \Hardy$ be arbitrary.
  Then
  \begin{equation}\label{eq:orth}
    F
    \perp
    \T_{\hH}\ManRatStab
  \end{equation}
  if and only if
  \begin{equation}\label{eq:interp}
    F\myparen*{-\overline{\lambda_i}} b_i = 0,\
    c_i\herm F\myparen*{-\overline{\lambda_i}} = 0,\
    c_i\herm F'\myparen*{-\overline{\lambda_i}} b_i = 0
  \end{equation}
  for $i = 1, 2, \ldots, r$.
\end{thm}
\begin{proof}
  Using Lemma~\ref{lem:tan-space-span},
  it follows that~\eqref{eq:orth} is equivalent to
  \begin{equation}\label{eq:orth-inner-prod}
    \begin{aligned}
      \ipHtwo*{\frac{e_j b_i\herm}{s - \lambda_i}}{F}
       & = 0,
       & j = 1, 2, \ldots, p,    \\
      \ipHtwo*{\frac{c_i e_{\ell}\herm}{s - \lambda_i}}{F}
       & = 0,
       & \ell = 1, 2, \ldots, m, \\
      \ipHtwo*{\frac{c_i b_i\herm}{\myparen*{s - \lambda_i}^2}}{F}
       & = 0,
    \end{aligned}
  \end{equation}
  for $i = 1, 2, \ldots, r$.
  Using Cauchy's integral formula gives that,
  for any $G \in \Hardy$, $\lambda \in \bbC^-$, $b \in \Cm$, and $c \in \Cp$
  (see, e.g.,~\cite[Lemma~2.1.4]{AntBG20}),
  \begin{align*}
    \ipHtwo*{\frac{c b\herm}{s - \lambda}}{G}
     & = c\herm G\myparen*{-\overline{\lambda}} b \quad \text{and} \\
    \ipHtwo*{\frac{c b\herm}{\myparen*{s - \lambda}^2}}{G}
     & = -c\herm G'\myparen*{-\overline{\lambda}} b.
  \end{align*}
  Therefore, it follows that~\eqref{eq:orth-inner-prod} is equivalent
  to~\eqref{eq:interp}.
\end{proof}
Applying Theorem~\ref{thm:orth-interp} to $F = H - \hH$ directly yields the
desired result.
\begin{cor}\label{cor:orth-interp-cond}
  Let $\hH \in \ManRatStab$,
  with a pole-residue form
  $\hH(s) = \sum_{i = 1}^r \frac{c_i b_i\herm}{s - \lambda_i}$,
  have pairwise distinct poles.
  Then the geometric necessary optimality
  condition~\eqref{eq:orth-tan-space-cond} is equivalent to the interpolatory
  necessary conditions~\eqref{eq:h2-opt-cond-interp}.
\end{cor}
Note that the above results extend to real systems.
In particular, in Lemma~\ref{lem:tan-space-span},
the spanning set~\eqref{eq:tan-space-span} of $\T_{\hH}\ReManRatStab$
becomes of twice the size by splitting the rational functions
in~\eqref{eq:tan-space-span} into real and imaginary parts.
Then Theorem~\ref{thm:orth-interp} remains the same after replacing
$\ManRatStab$ and $\Hardy$ by $\ReManRatStab$ and $\ReHardy$, respectively.
Therefore, also Corollary~\ref{cor:orth-interp-cond} stays the same.

\subsection{Geometric Interpretation of IRKA}%
\label{sec:irka-geom}

\Ac{irka} has been initially introduced and
studied as a fixed point iteration~\cite{GugAB08}.
In this section,
our goal is to analyze the geometry of the \ac{irka} iteration for the \ac{mimo}
case using the tools from \Cref{sec:geointerpolate}.

First, Theorem~\ref{thm:orth-interp} allows us to represent each \ac{irka}
step as an orthogonal projection as we prove next.
\begin{cor}\label{cor:irkastepsorth}
  Let $H \in \Hardy$ be given.
  Furthermore, let $\hH_k, \hH_{k + 1} \in \ManRatStab$ be two consecutive
  iterates of \ac{irka} such that $\hH_k$ has pairwise distinct poles.
  Then
  \begin{equation}\label{eq:irka-orth}
    H - \hH_{k + 1}
    \perp
    \T_{\hH_k}\ManRatStab,
  \end{equation}
  i.e., $\hH_{k + 1}$ is an orthogonal projection of $H$ onto $\ManRatStab$
  along the normal space of $\ManRatStab$ at $\hH_k$.
\end{cor}
\begin{proof}
  Applying Theorem~\ref{thm:orth-interp} with $F = H - \hH_{k + 1}$ shows
  that~\eqref{eq:irka-orth} is equivalent to~\eqref{eq:irka}.
\end{proof}
Thus, Corollary~\ref{cor:irkastepsorth} gives an interpretation of \ac{irka} as
an iterative orthogonal projection method.
This is illustrated in \Cref{fig:irka}.
In the next section,
this will help us in interpreting \ac{irka} as a Riemannian optimization method.

\begin{figure*}[tb]
  \centering
  \begin{subfigure}[t]{0.48\textwidth}
    \centering
    \begin{tikzpicture}[x=1em, y=1em, scale=1.2]
      \def\xmin{-10}
      \def\xmax{10}
      \def\coeff{-0.03}
      \def\Hr{-7}
      \def\H{5}

      \node at (-7, 6) {$\Hardy$};

      \coordinate (H) at (0, \H);
      \draw[fill] (H) circle (0.1) node[above] {$H$};

      \draw[domain=\xmin:\xmax, thick] plot (\x, \coeff*\x*\x);
      \node[rotate=-20] at (\xmax-3, \coeff*\xmax*\xmax+0.5) {$\ManRatStab$};

      \coordinate (Hr) at (\Hr, \coeff*\Hr*\Hr);
      \draw[fill] (Hr) circle (0.1) node[below] {$\hH_k$};

      \pgfmathsetmacro{\foo}{2*\coeff*\Hr*(\xmin - \Hr) + \coeff*\Hr*\Hr}
      \pgfmathsetmacro{\bar}{2*\coeff*\Hr*(\xmax - \Hr) + \coeff*\Hr*\Hr}
      \draw (\xmin, \foo) -- node[pos=0.9, above, sloped]
      {$\T_{\hH_k}\ManRatStab$} (\xmax, \bar);

      \pgfmathsetmacro{\newHrx}{
        (1 - sqrt(1 + 16*\coeff*\coeff*\coeff*\Hr*\Hr*\H))
        / (-4*\coeff*\coeff*\Hr)
      }
      \pgfmathsetmacro{\newHry}{\coeff*\newHrx*\newHrx}
      \coordinate (newHr) at (\newHrx, \newHry);
      \draw[fill] (newHr) circle (0.1) node[below] {$\hH_{k + 1}$};

      \draw[dashed] (H) -- (newHr);
    \end{tikzpicture}
    \caption{IRKA iteration.
      $\hH_{k + 1}$ is an orthogonal projection of $H$ along the orthogonal
      complement of $\T_{\hH_k} \ManRatStab$ onto $\ManRatStab$}%
    \label{fig:irka}
  \end{subfigure}
  \hfill
  \begin{subfigure}[t]{0.48\textwidth}
    \centering
    \begin{tikzpicture}[x=1em, y=1em, scale=1.2]
      \def\xmin{-10}
      \def\xmax{10}
      \def\coeff{-0.03}
      \def\Hr{-7}
      \def\H{5}

      \node at (-7, 6) {$\Hardy$};

      \coordinate (H) at (0, \H);
      \draw[fill] (H) circle (0.1) node[above] {$H$};

      \draw[domain=\xmin:\xmax, thick] plot (\x, \coeff*\x*\x);
      \node[rotate=-20] at (\xmax-3, \coeff*\xmax*\xmax+0.5) {$\ManRatStab$};

      \coordinate (Hr) at (\Hr, \coeff*\Hr*\Hr);
      \draw[fill] (Hr) circle (0.1) node[below] {$\hH_k$};

      \draw[-latex, ultra thick] (Hr)
      -- node[pos=0.6, above, sloped] {$-\nabla{\of}(\hH_k)$}
      (H);

      \pgfmathsetmacro{\foo}{2*\coeff*\Hr*(\xmin - \Hr) + \coeff*\Hr*\Hr}
      \pgfmathsetmacro{\bar}{2*\coeff*\Hr*(\xmax - \Hr) + \coeff*\Hr*\Hr}
      \draw (\xmin, \foo) -- node[pos=0.9, above, sloped]
      {$\T_{\hH_k}\ManRatStab$} (\xmax, \bar);

      \pgfmathsetmacro{\gfx}{
        2*\coeff*\Hr*(\H + \coeff*\Hr*\Hr)/(4*\coeff*\coeff*\Hr*\Hr + 1)
      }
      \pgfmathsetmacro{\gfy}{
        \coeff*\Hr*\Hr*(4*\coeff*\H - 1)/(4*\coeff*\coeff*\Hr*\Hr + 1)
      }
      \coordinate (gf) at (\gfx, \gfy);
      \draw[-latex, ultra thick] (Hr)
      -- node[pos=0.7, above, sloped] {$-\grad{f}(\hH_k)$}
      (gf);

      \pgfmathsetmacro{\newHrx}{
        (1 - sqrt(1 + 16*\coeff*\coeff*\coeff*\Hr*\Hr*\H))
        / (-4*\coeff*\coeff*\Hr)
      }
      \pgfmathsetmacro{\newHry}{\coeff*\newHrx*\newHrx}
      \coordinate (newHr) at (\newHrx, \newHry);
      \draw[fill] (newHr) circle (0.1) node[below] {$\hH_{k + 1}$};

      \draw[dashed] (H) -- (newHr);

      \node at (\newHrx + 12.5ex, \newHry + 1.5ex)
      {$\R_{\hH_k}\myparen*{-\grad{f}\myparen{\hH_k}}$};
    \end{tikzpicture}
    \caption{Riemannian gradient descent iteration
      for~\eqref{eq:h2-opt-prob-riemann}.
      $\nabla{\of}(\hH_k)$ is the Euclidean gradient,
      $\grad{f}(\hH_k)$ is the Riemannian gradient, and
      $\R_{\hH_k}$ is the orthographic retraction}%
    \label{fig:irka_riemann}
  \end{subfigure}
  \caption{Two methods for $\Htwo$-optimal model order reduction.
    $H$ is the full-order transfer function,
    $\ManRatStab$ is the manifold of stable rational functions of degree $r$,
    $\hH_k$ is the current reduced-order transfer function,
    $\T_{\hH_k} \ManRatStab$ is the tangent space of $\ManRatStab$ at $\hH_k$,
    and $\hH_{k + 1}$ is the next iterate}%
\end{figure*}
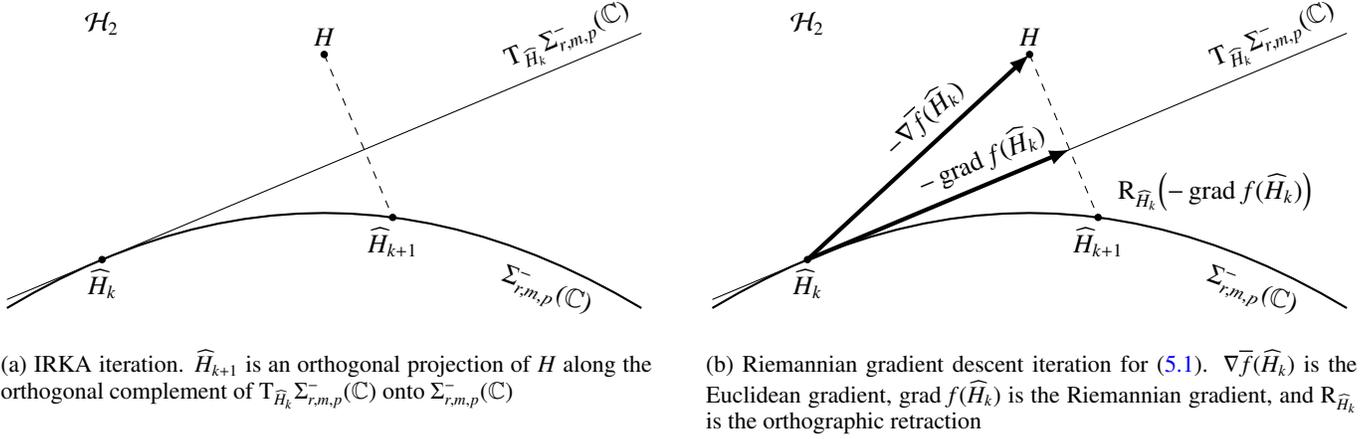

\section{IRKA is a Riemannian Optimization Method}%
\label{sec:irka-riemann}

So far, we have interpreted $\Htwo$-optimal \ac{mor} as
a Riemannian optimization problem and
gave a geometric interpretation of \ac{irka} based on
the embedded manifold structure of $\ManRatStab$.
With this, it becomes clear that \ac{irka} can be interpreted as
a Riemannian gradient descent method applied to~\eqref{eq:h2-opt-prob-riemann}
for a particular choice of step size and retraction.

Indeed, \Cref{fig:irka_riemann} suggests that \ac{irka} is a
Riemannian gradient descent method~\eqref{eq:rgd} with $\alpha_k = 1$ and
the orthographic retraction, i.e.,
a projection to the manifold that is orthogonal to the tangent space.
In \Cref{sec:retraction},
we extend a previous result on orthographic retractions from finite-dimensional
Euclidean spaces to infinite-dimensional spaces.
Then using this result in \Cref{sec:irka-rgd},
we show that \ac{irka} does behave as a Riemannian gradient descent method when
well-defined.

\subsection{Orthographic Retraction}%
\label{sec:retraction}

The orthographic retraction $\R$ is such that
\begin{equation*}
  \hH + G - \R_{\hH}(G) \perp \T_{\hH}\ManRatStab,
\end{equation*}
for all $\hH \in \ManRatStab$ and
all $G$ in some open subset of $\T_{\hH}\ManRatStab$ containing zero.
There is work on projection-like retractions,
including orthographic retractions,
in~\cite[Lemma~20]{AbsM12} for embedded submanifolds of finite-dimensional
Euclidean spaces.
The following lemma generalizes it to infinite-dimensional spaces.
\begin{lem}
  Let $\cM$ be an embedded submanifold of a Hilbert space $\bbE$.
  For all $\hp \in \cM$,
  there exists a neighborhood $\cU_{\T\cM}$ of $(\hp, 0)$ in $\T\cM$ such that
  for all $(p, v) \in \cU_{\T\cM}$,
  there is one and only one smallest $w \in \N_p\cM$ such that
  $p + v + w \in \cM$.
  Call it $w(p, v)$ and define $\R(p, v) = p + v + w(p, v)$.
  We have $\D_v w(p, 0) = 0$ and thus $\R$ defines a retraction around $\hp$.
  Since the expression of $\R$ does not depend on $\hp$ or $\cU_{\T\cM}$,
  $\R$ defines a retraction on $\cM$.
\end{lem}
\begin{proof}
  The proof of~\cite[Lemma~20]{AbsM12} works exactly the same in this setting,
  except that $\bbR^d$ and $\bbR^{n - d}$
  need to be replaced by $\bbE_1$ and $\bbE_2$, respectively,
  where $\bbE \simeq \bbE_1 \times \bbE_2$ such that for any $\hp \in \cM$ there
  exists a chart $\fundef{\varphi}{U}{\bbE_1 \times \bbE_2}$ of $\bbE$
  around $\hp$ such that
  $\varphi(p) \in \bbE_1 \times \{0\}$ if and only if $p \in \cM \cap U$
  (see~\cite[Theorem~1.3.5]{Kli11}).
\end{proof}

\subsection{IRKA as Riemannian Gradient Descent}%
\label{sec:irka-rgd}

We can now show that \ac{irka} is a Riemannian gradient descent method,
assuming \ac{irka} iterates are well-defined.
\begin{thm}\label{thm:irka-rgd}
  Let $\hH_k, \hH_{k + 1} \in \ManRatStab$ be two consecutive iterates of
  \ac{irka},
  $f$ as in~\eqref{eq:h2-opt-prob-riemann}, and
  $\fundef{\R}{\T\ManRatStab}{\ManRatStab}$ the orthographic retraction.
  Furthermore, let $\hH_{k + 1}$ be the unique element of $\ManRatStab$
  satisfying~\eqref{eq:irka-orth} and
  let $\R_{\hH_k}\myparen{-\grad f(\hH_k)}$ exist.
  Then
  \begin{equation}\label{eq:irkastep}
    \hH_{k + 1} = \R_{\hH_k}\myparen*{-\grad f(\hH_k)}.
  \end{equation}
\end{thm}
\begin{proof}
  Since $\R$ is the orthographic retraction, we have
  \begin{equation*}
    \hH_k
    - \grad f(\hH_k)
    - \R_{\hH_k}\myparen*{-\grad f(\hH_k)}
    \perp
    \T_{\hH_k}\ManRatStab.
  \end{equation*}
  Next, the relation~\eqref{eq:riemann-grad-h2} can be written as
  \begin{equation*}
    \grad f(\hH_k)
    + H - \hH_k
    \perp
    \T_{\hH_k}\ManRatStab.
  \end{equation*}
  Adding the above two relations, we find
  \begin{equation*}
    H
    - \R_{\hH_k}\myparen*{-\grad f(\hH_k)}
    \perp
    \T_{\hH_k}\ManRatStab.
  \end{equation*}
  Since $\hH_{k + 1}$ satisfying~\eqref{eq:irka-orth} was assumed to be unique,
  it follows that $\hH_{k + 1}$ and $\R_{\hH_k}\myparen{-\grad f(\hH_k)}$
  are equal.
\end{proof}
This shows that \ac{irka} is a Riemannian gradient descent method with a fixed
step size.
The qualifier ``well-defined'' is needed as some iterates of \ac{irka} may not
lie on the manifold $\ManRatStab$ (they may be unstable or of lower order)
or there may not be a unique interpolant~\eqref{eq:irka}.

\section{Algorithmic Developments}%
\label{sec:algorithm}

Now that we have established the understanding of \ac{irka} as
a Riemannian gradient descent method with a fixed step size,
we are in a position to improve its performance and develop new algorithms
by considering various Riemannian optimization techniques.

The first such approach would be to implement a variation of \ac{irka} using
Riemannian gradient descent with variable step size.
In particular, this is relevant because of two issues:
\begin{enumerate}
  \item the orthographic retraction may only be defined locally,
  \item $\ReManRatStab$ is disconnected for $m = p = 1$.
\end{enumerate}
Regarding the first point,
it is known that \ac{irka} may generate intermediate iterates with
unstable poles.
For the second point,
we find that \ac{irka} can jump between different connected components,
but the theory of Riemannian optimization assumes that the manifold is
connected.
We resolve both of these issues using backtracking line search,
as discussed in this section.

We discuss connectedness and how to detect connected components in
\Cref{sec:connectedness}.
We emphasize that connectedness is a potential concern only in the \ac{siso}
case.
In \Cref{sec:step-size},
we employ backtracking line search to enforce stability,
a decrease in $\Htwo$ error, and
that the iterates remain on the same connected component.

\subsection{Connectedness}%
\label{sec:connectedness}

As mentioned in~\Cref{sec:manifold},
$\ReManRatStab$ has $r + 1$ connected components when $m = p = 1$.
Furthermore, the connected components can be identified by the Cauchy index.

For a real proper rational function with partial fraction expansion
\begin{equation*}
  F(s) =
  \sum_{i = 1}^k
  \sum_{j = 1}^{\mu_i}
  \frac{\varphi_i^{(j)}}{{(s - \lambda_i)}^j}
  + G(s),
\end{equation*}
where $\lambda_1, \ldots, \lambda_k$ are all real distinct poles of $F$
($G$ has no real poles),
the Cauchy index of $F$, colloquially speaking,
is the number of real poles at which the function jumps from $-\infty$ to
$\infty$ minus the number of real poles with jumps from $\infty$ to $-\infty$.
Formally, the Cauchy index of $F$ is defined as
\begin{equation*}
  \cauchy(F) =
  \sum_{i = 1}^k
  \sum_{\substack{j = 1 \\ j \text{ odd}}}^{\mu_i}
  \sign*{\varphi_i^{(j)}}.
\end{equation*}
Assuming moderate $r$,
one can compute the pole-residue form of $\hH_k$ in a numerically efficient way
and compute $\cauchy(\hH_k)$ at every iteration in \ac{irka}.
Doing precisely that,
we show via various numerical experiments in~\Cref{sec:numerics} that,
for \ac{siso} systems,
during the \ac{irka} iterations,
the Cauchy index can change.
Using backtracking line search as we explain in the next sections,
we ensure that the Cauchy index remains the same throughout the iteration for
\ac{siso} systems.
We re-emphasize that connectedness is not a concern when
$m > 1$ and/or $p > 1$.

\subsection{Line Search}%
\label{sec:step-size}

Here we derive a Riemannian gradient descent method with variable step size and
propose a backtracking line search for $\Htwo$-optimal \ac{mor}.
This will resolve both potential issues listed at the beginning of this section.

Recall that Theorem~\ref{thm:irka-rgd} states that \ac{irka} is a Riemannian
gradient descent method with a fixed step size $\alpha_k = 1$, i.e.,
$\hH_{k + 1} = \R_{\hH_k}\myparen{-\grad f(\hH_k)}$
as shown in~\eqref{eq:irkastep}
where $\R$ denotes the orthographic retraction.
Here, exploiting this new Riemannian optimization perspective of \ac{irka},
we develop a new formulation where the classical \ac{irka} step is replaced with
$\hH_{k + 1} = \R_{\hH_k}\myparen{-\alpha_k \grad f(\hH_k)}$
for a positive step size $\alpha_k$.
With an appropriate choice of $\alpha_k$ with backtracking,
this modification guarantees stability of $\hH_{k}$ at every step and
also makes sure that the new modified Riemannian-based \ac{irka} iterate
do not jump between different connected components.

Two issues remain.
Recall that in the classical \ac{irka} case
($\hH_{k + 1} = \R_{\hH_k}\myparen{-\grad f(\hH_k)}$),
every iterate is obtained via bitangential Hermite interpolation.
How can we compute the iterates with a variable step size?
First, we answer this question and
show that the new updates can still be achieved using interpolation.
Also recall that at every \ac{irka} step,
one performs a Petrov-Galerkin projection as in~\eqref{eq:pg}.
However, a line search with backtracking will require re-computing $\hH_{k + 1}$
for all the $\alpha_k$'s tested.
A naive implementation of this will then require repeated Petrov-Galerkin
projections during line search,
which is computationally very expensive for the large-scale problems of interest
to \ac{mor}.
Thus, as a second contribution in this section,
we develop an efficient implementation for constructing $\hH_{k + 1}$
that avoids repeated Petrov-Galerkin projection during line search.

The next result shows that $\hH_{k + 1}$ obtained via
a step size $\alpha_k \neq 1$ is still an interpolant,
more precisely it is a bitangential Hermite interpolant of
$H_{k, \alpha_k} = (1 - \alpha_k) \hH_k + \alpha_k H$
(\Cref{fig:irka_riemann} helps to show why this is the case).
\begin{thm}\label{thm:step-orth}
  Let
  $H \in \Hardy$,
  $\hH_k \in \ManRatStab$,
  $\alpha_k > 0$.
  Let
  \[\hH_{k + 1} = \R_{\hH_k}\myparen*{-\alpha_k \grad f(\hH_k)}\]
  be obtained via Riemannian gradient descent with step size $\alpha_k$.
  Define
  \[H_{k, \alpha_k} = (1 - \alpha_k) \hH_k + \alpha_k H.\]
  Then \[H_{k, \alpha_k} - \hH_{k + 1} \perp \T_{\hH_k}\ManRatStab,\]
  thus $\hH_{k + 1}$ is a bitangential Hermite interpolant to $H_{k, \alpha_k}$.
\end{thm}
\begin{proof}
  From the property of $\Proj_{\hH_k}$, we have that
  \begin{equation*}
    H_{k, \alpha_k}
    - \hH_k
    - \Proj_{\hH_k}\myparen*{H_{k, \alpha_k} - \hH_k}
    \perp \T_{\hH_k}\ManRatStab.
  \end{equation*}
  Using that $\R$ is the orthographic retraction, we have
  \begin{equation*}
    \hH_k
    - \alpha_k \grad{f}(\hH_k)
    - \hH_{k + 1}
    \perp \T_{\hH_k}\ManRatStab.
  \end{equation*}
  Next, we observe that
  \begin{align*}
     & \Proj_{\hH_k}\myparen*{H_{k, \alpha_k} - \hH_k}
    = \Proj_{\hH_k}\myparen*{\alpha_k \myparen*{H - \hH_k}} \\
     & = \alpha_k \Proj_{\hH_k}\myparen*{H - \hH_k}
    = -\alpha_k \grad{f}(\hH_k).
  \end{align*}
  Therefore,
  \begin{align*}
    H_{k, \alpha_k} - \hH_{k + 1}
     & =
    \myparen*{
      H_{k, \alpha_k}
      - \hH_k
      - \Proj_{\hH_k}\myparen*{H_{k, \alpha_k} - \hH_k}
    }         \\*
     & \qquad
    +
    \myparen*{
      \hH_k
      - \alpha_k \grad{f}(\hH_k)
      - \hH_{k + 1}
    }
  \end{align*}
  is orthogonal to $\T_{\hH_k}\ManRatStab$.
\end{proof}
This result shows that $\hH_{k + 1}$ is a bitangential Hermite interpolant to
$H_{k, \alpha_k}$ and
thus can be constructed, e.g.,
using Petrov-Galerkin projection as discussed in~\Cref{sec:bhi}
but now applied to the affine combination
$H_{k, \alpha_k} = (1 - \alpha_k) \hH_k + \alpha_k H$
(as opposed to $H$ as done in classical \ac{irka}).
A naive implementation of line search would then,
for every new $\alpha_k$,
build the state-space form of $H_{k, \alpha_k}$,
compute the projection matrices by solving $2 r$ linear systems
in~\eqref{eq:spanVW},
and then project to find the new \ac{rom}, which would be costly.
The following theorem show how to avoid these repeated projections
during line search.
\begin{thm}
  Let $H(s) = C {(s E - A)}^{-1} B$ and
  $\hH_k(s) = \hC_k {(s \hE_k - \hA_k)}^{-1} \hB_k$
  be real systems.
  Then a state-space realization of
  $\hH_{k + 1} = \R_{\hH_k}\myparen*{-\alpha_k \grad f(\hH_k)}$ is
  \begin{subequations}\label{eq:step-proj}
    \begin{align}
      \hE_{k + 1}
       & =
      \hE_k
      - \alpha_k
      \myparen*{
        \hE_k
        - \hQ_k^{-1} \tQ_k\tran E \tP_k \hP_k^{-1}
      },   \\*
      \hA_{k + 1}
       & =
      \hA_k
      - \alpha_k
      \myparen*{
        \hA_k
        - \hQ_k^{-1} \tQ_k\tran A \tP_k \hP_k^{-1}
      },   \\*
      \hB_{k + 1}
       & =
      \hB_k
      - \alpha_k
      \myparen*{
        \hB_k
        - \hQ_k^{-1} \tQ_k\tran B
      },   \\*
      \hC_{k + 1}
       & =
      \hC_k
      - \alpha_k
      \myparen*{
        \hC_k
        -  C \tP_k \hP_k^{-1}
      },
    \end{align}
  \end{subequations}
  for any $\alpha_k > 0$
  where $\tP_k$, $\hP_k$, $\tQ_k$, $\hQ_k$ are solutions to Sylvester and
  Lyapunov equations
  \begin{subequations}\label{eq:gramians}
    \begin{align}
      A \tP_k \hE_k\tran
      + E \tP_k \hA_k\tran
      + B \hB_k\tran
       & = 0, \\
      \hA_k \hP_k \hE_k\tran
      + \hE_k \hP_k \hA_k\tran
      + \hB_k \hB_k\tran
       & = 0, \\
      A\tran \tQ_k \hE_k
      + E\tran \tQ_k \hA_k
      + C\tran \hC_k
       & = 0, \\
      \hA_k\tran \hQ_k \hE_k
      + \hE_k\tran \hQ_k \hA_k
      + \hC_k\tran \hC_k
       & = 0.
    \end{align}
  \end{subequations}
\end{thm}
\begin{proof}
  A state-space realization of $H_{k, \alpha_k}$ for $\alpha_k \neq 1$ is
  \begin{subequations}\label{eq:affine-sys}
    \begin{align}
      E_{k, \alpha_k}
       & =
      \begin{bmatrix}
        \alpha_k E & 0                    \\
        0          & (1 - \alpha_k) \hE_k
      \end{bmatrix}\!, \\*
      A_{k, \alpha_k}
       & =
      \begin{bmatrix}
        \alpha_k A & 0                    \\
        0          & (1 - \alpha_k) \hA_k
      \end{bmatrix}\!, \\*
      B_{k, \alpha_k}
       & =
      \begin{bmatrix}
        \alpha_k B \\
        (1 - \alpha_k) \hB_k
      \end{bmatrix}\!,               \\*
      C_{k, \alpha_k}
       & =
      \begin{bmatrix}
        \alpha_k C & (1 - \alpha_k) \hC_k
      \end{bmatrix}\!.
    \end{align}
  \end{subequations}
  Then $\hH_{k + 1}$ can be realized by
  \begin{equation}\label{eq:step-proj-VW}
    \begin{alignedat}{4}
      \hE_{k + 1} & = W\tran E_{k, \alpha_k} V, & \quad
      \hA_{k + 1} & = W\tran A_{k, \alpha_k} V, \\*
      \hB_{k + 1} & = W\tran B_{k, \alpha_k}, & \quad
      \hC_{k + 1} & = C_{k, \alpha_k} V,
    \end{alignedat}
  \end{equation}
  where $V$ and $W$ are solutions to Sylvester equations
  (see \Cref{sec:bhi})
  \begin{subequations}\label{eq:forVandW}
    \begin{align}
      A_{k, \alpha_k} V \hE_k\tran
      + E_{k, \alpha_k} V \hA_k\tran
      + B_{k, \alpha_k} \hB_k\tran
       & = 0, \\
      A_{k, \alpha_k}\tran W \hE_k
      + E_{k, \alpha_k}\tran W \hA_k
      + C_{k, \alpha_k}\tran \hC_k
       & = 0.
    \end{align}
  \end{subequations}
  Exploiting the block structure~\eqref{eq:affine-sys}
  of the coefficient matrices in~\eqref{eq:forVandW} and
  comparing it to~\eqref{eq:gramians},
  we find that
  \(V =
  \begin{bsmallmatrix}
    \tP_k \\
    \hP_k
  \end{bsmallmatrix}
  \)
  and
  \(W =
  \begin{bsmallmatrix}
    \tQ_k \\
    \hQ_k
  \end{bsmallmatrix}.
  \)
  Therefore, the expressions in~\eqref{eq:step-proj-VW} simplify to
  \begin{align*}
    \hE_{k + 1}
     & =
    \alpha_k \tQ_k\tran E \tP_k
    + (1 - \alpha_k) \hQ_k \hE_k \hP_k, \\*
    \hA_{k + 1}
     & =
    \alpha_k \tQ_k\tran A \tP_k
    + (1 - \alpha_k) \hQ_k \hA_k \hP_k, \\*
    \hB_{k + 1}
     & =
    \alpha_k \tQ_k\tran B
    + (1 - \alpha_k) \hQ_k \hB_k,       \\*
    \hC_{k + 1}
     & =
    \alpha_k C \tP_k
    + (1 - \alpha_k) \hC_k \hP_k,
  \end{align*}
  which also holds for $\alpha_k = 1$.
  Premultiplying by $\hQ_k^{-1}$ and postmultiplying by $\hP_k^{-1}$
  gives the equivalent realization of $\hH_{k + 1}$ in~\eqref{eq:step-proj}.
\end{proof}

We now have all the ingredients to design an effective numerical algorithm for
a variation of \ac{irka} using Riemannian gradient descent with backtracking.
A pseudo-code is presented in \Cref{alg:irka2}.
We initialize the method with
a stable \ac{rom} $\hH_1$ and
construct a new \ac{rom} with $\alpha_k = 1$, i.e.,
a step of \ac{irka}.
Then, we do backtracking and
reduce the step size by two
if $\hH_{k + 1}$ is unstable or
the $\Htwo$ has increased, or
the Cauchy index has changed
(in the \ac{siso} case only).
The algorithm terminates
if the step size becomes smaller than
a specified minimum step size $\alpha_{\min} > 0$ or
if a desired tolerance is reached.
Thus, \Cref{alg:irka2} guarantees stability throughout,
the $\Htwo$ error decreases at every step, and
upon convergence (interpolatory) $\Htwo$-optimality conditions are satisfied.

\begin{algorithm}
  \caption{\acf{irka2}}%
  \label{alg:irka2}
  \begin{algorithmic}[1]
    \Require%
    \Ac{fom} $(E, A, B, C)$,
    initial \ac{rom} $(\hE_1, \hA_1, \hB_1, \hC_1)$,
    maximum iteration number $\mathtt{maxit}$,
    minimum step size $\alpha_{\min} > 0$,
    tolerance $\mathtt{tol} > 0$.
    \Ensure%
    \Ac{rom} $(\hE, \hA, \hB, \hC)$.
    \State%
    Set $\hH_1$ as the transfer function of the initial \ac{rom}.
    \For{$k$ in $1, 2, \ldots, \mathtt{maxit}$}
    \State%
    $\alpha_k = 1$
    \While{\ac{rom} in~\eqref{eq:step-proj} is unstable,
      has a different Cauchy index, or
      it increases the $\Htwo$ error}
    \State%
    $\alpha_k = \alpha_k / 2$
    \If{$\alpha_k < \alpha_{\min}$}%
    \State%
    Exit the \textbf{for} loop.
    \EndIf%
    \EndWhile%
    \State%
    Set $(\hE_{k + 1}, \hA_{k + 1}, \hB_{k + 1}, \hC_{k + 1})$ as
    in~\eqref{eq:step-proj} with transfer function $\hH_{k + 1}$.
    \If{$\normHtwo{\hH_k - \hH_{k + 1}}
        \le \mathtt{tol} \cdot \alpha_k \normHtwo{\hH_{k + 1}}$}%
    \State%
    Exit the \textbf{for} loop.
    \EndIf%
    \EndFor%
    \State%
    Return the last computed \ac{rom}.
  \end{algorithmic}
\end{algorithm}

A couple of remarks are in order.
In checking whether $\Htwo$ error has decreased, we
do \emph{not} explicitly compute $\normHtwo{H - \hH_k}$
as this will require solving a large-scale Lyapunov equation at every step.
Instead note that
\begin{equation*}
  \normHtwo*{H - \hH_k}^2
  = \normHtwo{H}^2
  + 2 \Real*{\ipHtwo*{H}{\hH_k}}
  + \normHtwo*{\hH_k}^2.
\end{equation*}
The first term in this formula is a constant.
Thus, all we need to check is the last two terms
which can be easily and effectively computed.
We chose a stopping criterion based on a relative change in the \ac{rom}.
Other stopping criteria are also possible.

\begin{rem}
  We note that computations in~\eqref{eq:step-proj}
  involve solving linear systems with reduced Gramians $\hP_k$ and $\hQ_k$
  as system matrices.
  We observed that in some cases
  these reduced Gramians can be ill-conditioned,
  which then leads to choosing small step sizes,
  thus requiring small $\alpha_{\min}$.
  In this paper, our main goal is to establish the main theoretical framework.
  Developing better ways of handling such numerical issues, e.g.
  by regularizing the Gramians via truncated SVD,
  would be a topic of future research.
  We have also observed that in some cases $\hE_k$
  can become  ill-conditioned.
  As a way to circumvent this, we check the condition number of $\hE_k$ and
  if it becomes large (e.g., larger than $10^4$,
  which we use in the numerical examples),
  we apply a state-space transform to make $\hE_k$ the identity matrix.
\end{rem}

\subsection{Interpretation as a Fixed Point Iteration}

Note that Theorem~\ref{thm:step-orth},
together with Theorem~\ref{thm:orth-interp},
gives that $\hH_{k + 1}$ is a bitangential Hermite interpolant of
$H_{k, \alpha_k}$ at the reflected poles and in the directions of
$\hH_k(s) = \sum_{i = 1}^r \frac{\cik \bikstar}{s - \lamik}$, i.e.,
\begin{subequations}\label{eq:irka-alpha}
  \begin{align}
    H_{k, \alpha_k}\myparen*{-\lamikstar} \bik
     & =
    \hH_{k + 1}\myparen*{-\lamikstar} \bik,     \\
    \cikstar H_{k, \alpha_k}\myparen*{-\lamikstar}
     & =
    \cikstar \hH_{k + 1}\myparen*{-\lamikstar}, \\
    \cikstar H_{k, \alpha_k}''\myparen*{-\lamikstar} \bik
     & =
    \cikstar \hH_{k + 1}'\myparen*{-\lamikstar} \bik,
  \end{align}
\end{subequations}
for $i = 1, 2, \ldots, r$.
This can be interpreted as a fixed point iteration applied to
\begin{align*}
  \myparen*{\alpha H + (1 - \alpha) \hH}\myparen*{-\overline{\lambda_i}}
  b_i
   & =
  \hH\myparen*{-\overline{\lambda_i}}
  b_i,                                 \\
  c_i\herm
  \myparen*{\alpha H + (1 - \alpha) \hH}\myparen*{-\overline{\lambda_i}}
   & =
  c_i\herm
  \hH\myparen*{-\overline{\lambda_i}}, \\
  c_i\herm
  \myparen*{\alpha H + (1 - \alpha) \hH}'\myparen*{-\overline{\lambda_i}}
  b_i
   & =
  c_i\herm
  \hH'\myparen*{-\overline{\lambda_i}}
  b_i,
\end{align*}
which is clearly equivalent to~\eqref{eq:h2-opt-cond-interp}
for any $\alpha > 0$.
Therefore, the proposed Riemannian gradient descent formulation with line search
for $\Htwo$ minimization can be considered as a fixed point iteration applied to
modified interpolatory necessary optimality conditions;
thus more clearly illustrating the distinction from the original \ac{irka}
formulation, which corresponds to $\alpha = 1$.

\section{Numerical Examples}%
\label{sec:numerics}

The code producing the presented results is available at~\cite{Mli23},
written in the Python programming language and based on pyMOR~\cite{MilRS16}.
In all examples, we set $\mathtt{maxit} = 100$, $\alpha_{\min} = 10^{-20}$, and
$\mathtt{tol} = 10^{-4}$ in \Cref{alg:irka2}.

\subsection{Example 1}

We start with the following example from~\cite[Section~5.4]{GugAB08}:
\begin{equation*}
  H(s) = \frac{-s^2 + (7/4) s + 5/4}{s^3 + 2 s^2 + (17/16) s + 15/32}.
\end{equation*}
Notably,~\cite{GugAB08} shows that for the reduced order $r = 1$,
the globally $\Htwo$-optimal \ac{rom} for $H$,
with the pole at $0.27272$,
is repellent for \ac{irka},
thus \ac{irka} does not converge.
We apply \ac{irka} and the proposed method, i.e. \acs{irka2}, as implemented in
\Cref{alg:irka2}, to this model.
\Cref{fig:gab3_r1_c1} shows the results when both method are initialized with
$\hE_1 = 1$, $\hA_1 = -0.27$, $\hB_1 = 1$, $\hC_1 = 1$,
which has a Cauchy index of $1$ and
it is close to the global minimum
(as used in~\cite{GugAB08}).
Indeed, we observe that \ac{irka} iterates move away from the optimum as
expected since the minimizer is a repellent fixed point.
And note that many intermediate \ac{irka} iterates are unstable.
Furthermore, the Cauchy index changes many times during the \ac{irka}
iterations.
On the other hand, \acs{irka2} converges quickly
(with stability guarantee at every step) and
it often uses $\alpha_k < 1$ as the step size.
Moreover, note that the Cauchy index stays constant during \acs{irka2}.

\begin{figure}[tb]
  \centering
  \includegraphics[width=24em]{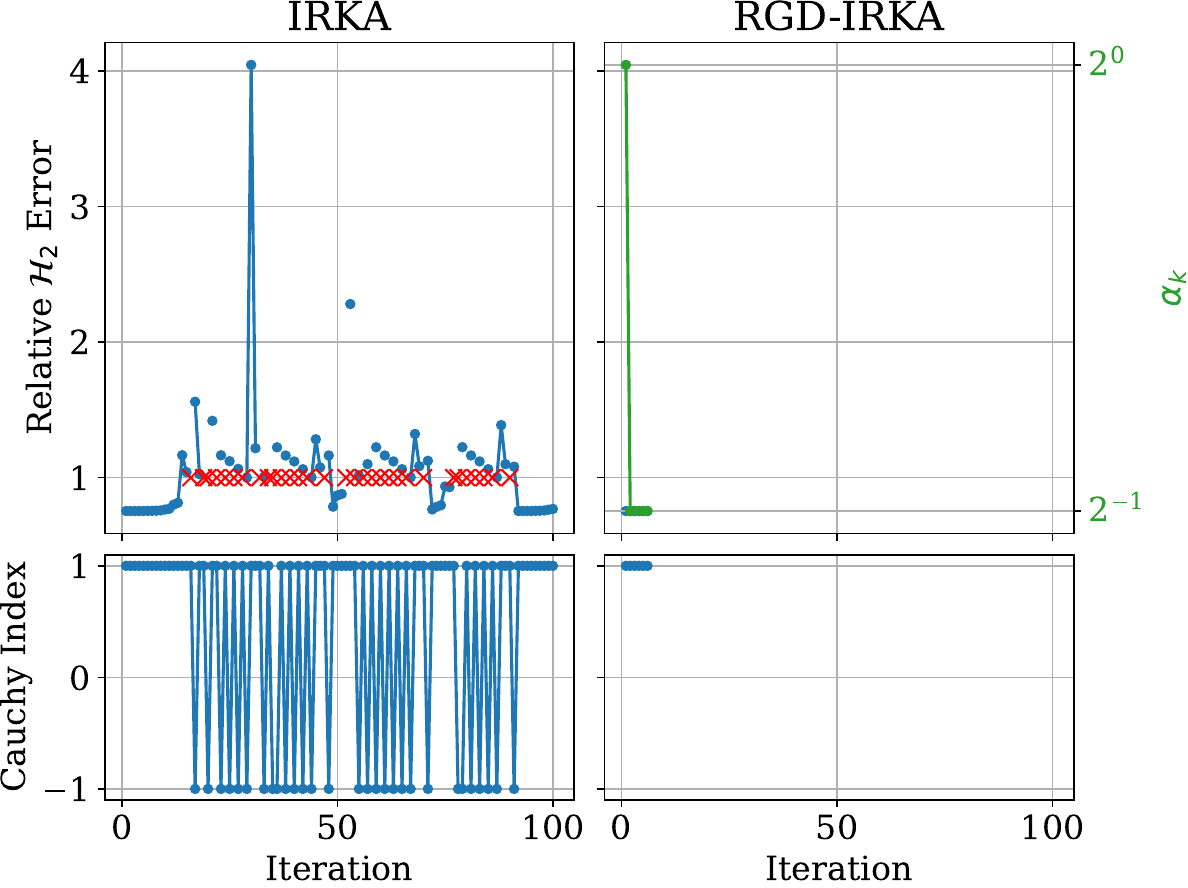}
  \caption{IRKA and \acs{irka2} results for the~\cite{GugAB08} example and
    $r = 1$.
    Red crosses represent unstable models}%
  \label{fig:gab3_r1_c1}
\end{figure}

Next we compare \ac{irka} and \acs{irka2} on the same example for $r = 2$
in~\Cref{fig:gab3_r2_c0},
initialized with
\begin{equation*}
  \hE_1 = I_2,\
  \hA_1 =
  \begin{bmatrix}
    -1 & 1  \\
    -1 & -1
  \end{bmatrix}\!,\
  \hB_1 =
  \begin{bmatrix}
    1 \\
    1
  \end{bmatrix}\!,\
  \hC_1 =
  \begin{bmatrix}
    1 & 1
  \end{bmatrix}\!.
\end{equation*}
Note that its transfer function has Cauchy index of $0$.
In this case with this specific initialization,
\ac{irka} and \acs{irka2} behave exactly the same.
So, for these specific choices,
the original \ac{irka} steps indeed correspond to descent steps.

\begin{figure}[tb]
  \centering
  \includegraphics[width=24em]{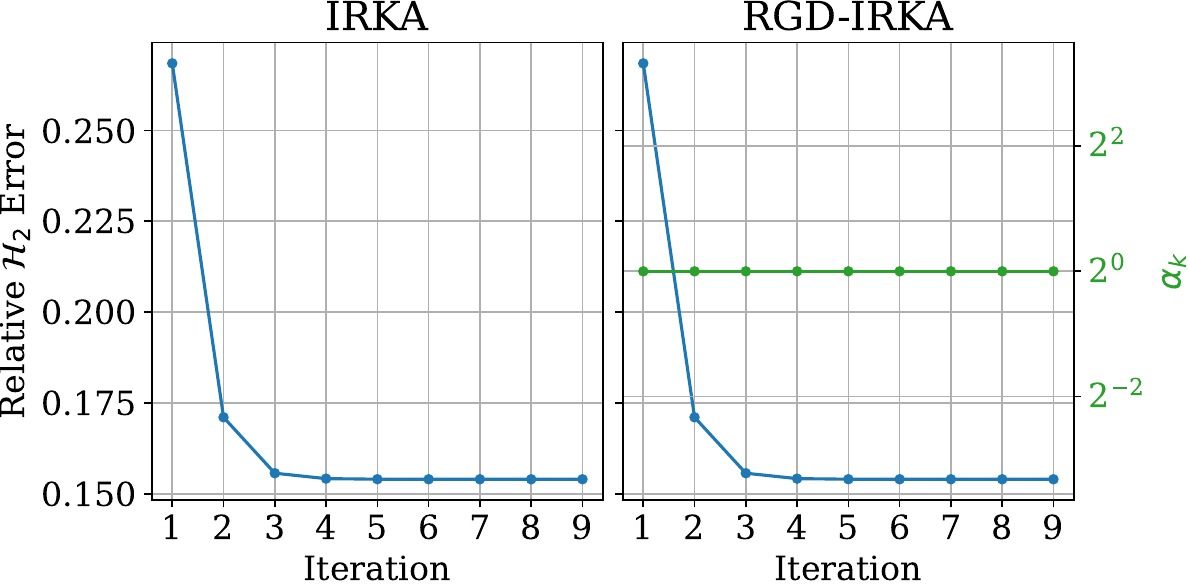}
  \caption{IRKA and \acs{irka2} results for the~\cite{GugAB08} example, $r = 2$,
    and initial Cauchy index of $0$}%
  \label{fig:gab3_r2_c0}
\end{figure}

Next for $r = 2$, we change the initialization to
\begin{equation*}
  \hE_1 = I_2,\
  \hA_1 =
  \begin{bmatrix}
    -1 & 0  \\
    0  & -2
  \end{bmatrix}\!,\
  \hB_1 =
  \begin{bmatrix}
    1 \\
    1
  \end{bmatrix}\!,\
  \hC_1 =
  \begin{bmatrix}
    1 & 1
  \end{bmatrix}\!.
\end{equation*}
which has Cauchy index of $2$.
The results, depicted in \Cref{fig:gab3_r2_c2},
show that, in the \ac{siso} case,
where the manifold $\ReManRatStab$ is disconnected and
\ac{irka2} enforces the iterates to stay on the same connected component,
a bad choice of the initial Cauchy index
can lead to a bad \ac{rom} for \acs{irka2},
but \ac{irka} still converges to a good local minimum as it jumps to a model
with a different Cauchy index.
Furthermore, we note that the pole-residue form of the \ac{rom} from \ac{irka2}
is
\begin{equation*}
  \frac{0.97188}{s + 0.27344}
  + \frac{8.4933}{s + 6.9933 \times 10^{12}}.
\end{equation*}
It suggests that,
since \acs{irka2} is restricted to one connected component due to line search,
the iterates converge to the boundary of the connected component,
which are systems of lower order.
We emphasize that Cauchy index initialization is only a potential concern in the
\ac{siso} case.

\begin{figure}[tb]
  \centering
  \includegraphics[width=24em]{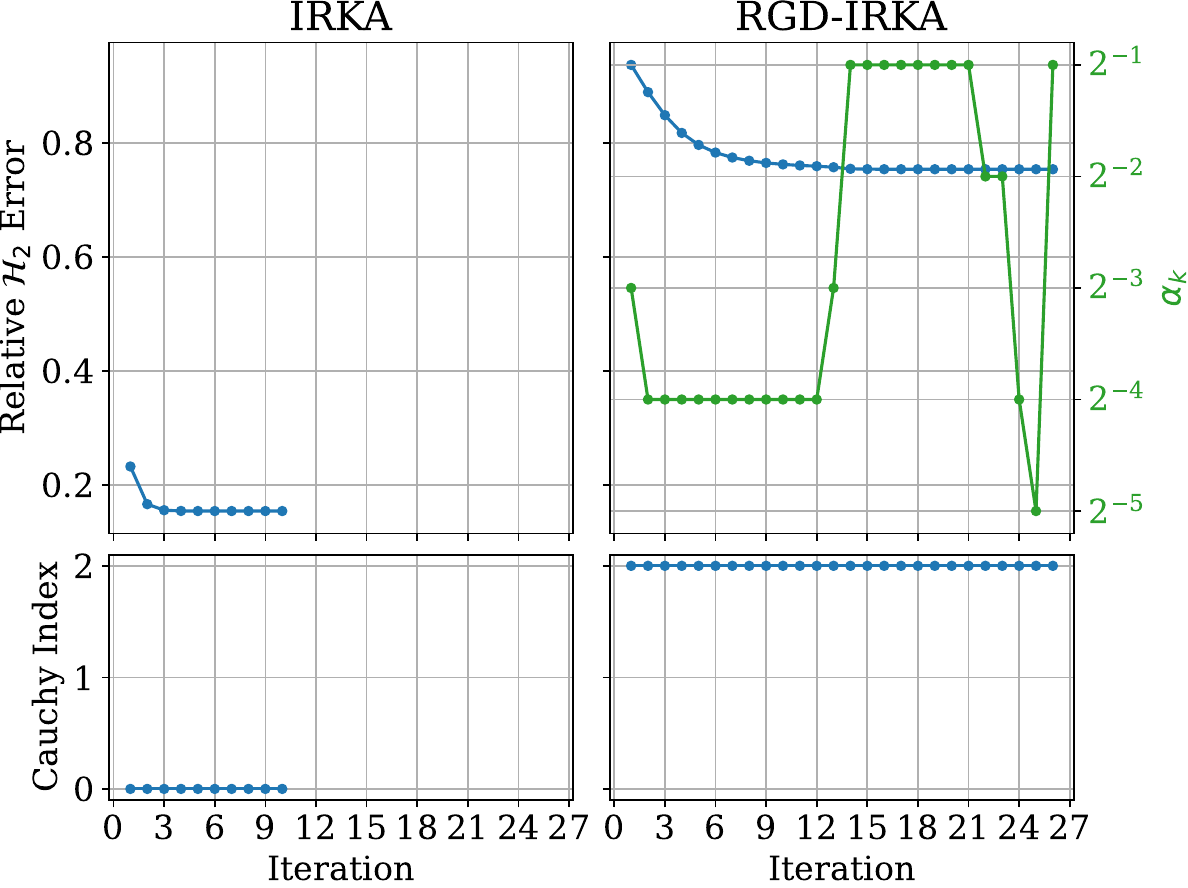}
  \caption{IRKA and \acs{irka2} results for the~\cite{GugAB08} example, $r = 2$,
    and initial Cauchy index of $2$}%
  \label{fig:gab3_r2_c2}
\end{figure}

\subsection{Example 2}

Here we show results on the CD player example from the NICONET benchmark
collection~\cite{ChaV02},
which is \iac{mimo} system with $2$ inputs and $2$ outputs.
\Cref{fig:cd_r6} shows the result for $r = 6$, initialized with
$\hE_1 = I_6$,
$\hA_1 = \mydiag{-1, -2, \dots, -6}$,
$\hB_1 = \ones_{6 \times 2}$,
$\hC_1 = \ones_{2 \times 6}$,
where $\ones_{m \times n} \in \bbR^{m \times n}$ is a matrix of ones.
Note that, since the manifold $\ReManRatStab$ is now connected,
Cauchy index is not necessary (and is not even defined).
Here we observe convergence of \ac{irka} with most iterates being unstable.
On the other hand, \acs{irka2} converges faster and
reaches a slightly better \ac{rom}
(relative $\Htwo$ errors are $1.9003 \times 10^{-3}$ for \ac{irka} and
$1.1167 \times 10^{-3}$ for \ac{irka2}).
The first step size in \acs{irka2} is $2^{-14} \approx 6 \times 10^{-5}$
while the later ones are $1$.

\begin{figure}[tb]
  \centering
  \includegraphics[width=24em]{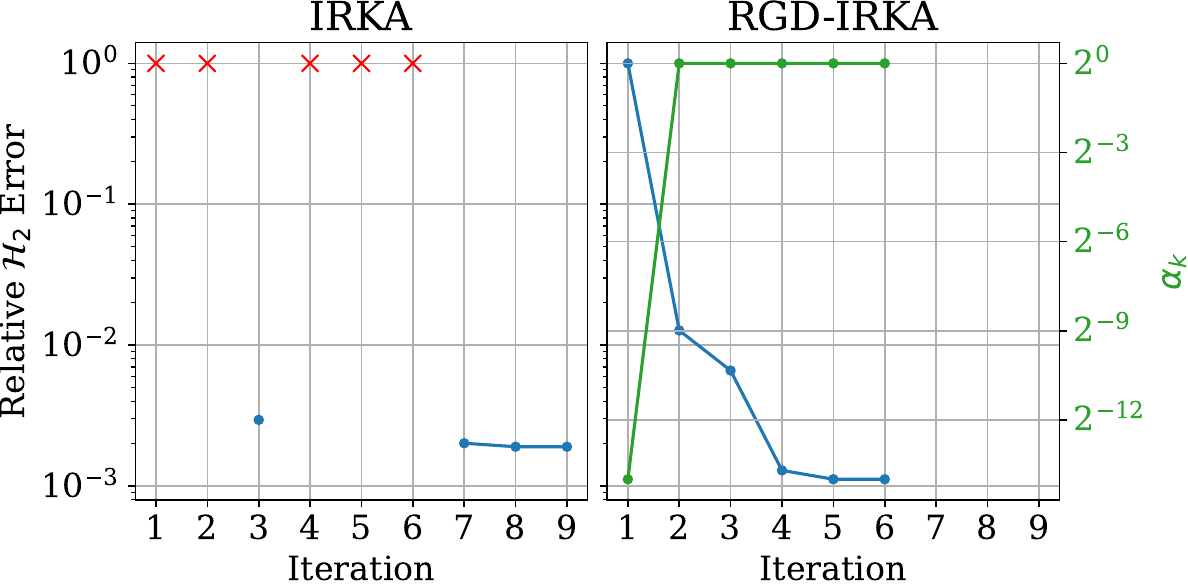}
  \caption{IRKA and \acs{irka2} results for the CD player example and $r = 6$.
    Red crosses represent unstable models}%
  \label{fig:cd_r6}
\end{figure}

\section{Conclusion}%
\label{sec:conclusion}

We showed that \ac{irka} can be interpreted as a Riemannian gradient descent
method with a fixed step size.
This interpretation directly leads to a development of a Riemannian gradient
descent formulation of \ac{irka}, called \ac{irka2},
that employs variable step size and line search,
with guaranteed convergence and stability preservation.
Numerical examples illustrated the benefits of~\ac{irka2}.



\bibliographystyle{alphaurl}
\bibliography{my}
\addcontentsline{toc}{chapter}{References}
\end{document}